\documentclass{amsart} 

\usepackage{amsmath}
\usepackage{amsthm}
\usepackage{amssymb}
\usepackage{amsfonts}
\usepackage{amscd}

\newcommand\rank{\mathop{\rm rank}\nolimits}
\newcommand\im{\mathop{\rm Im}\nolimits}
\newcommand\coker{\mathop{\rm coker}\nolimits}

\newcommand{\Tr}{\mathop{\rm Tr}\nolimits}

\newcommand\Hom{\mathop{\rm Hom}\nolimits}

\newcommand\Spec{\mathop{\rm Spec}\nolimits}

\newcommand{\length}{\mathop{\rm length}\nolimits}
\newcommand{\res}{\mathop{\sf res}\nolimits}
\newcommand\Quot{\mathop{\rm Quot}\nolimits}

\newcommand{\balpha}{\boldsymbol \alpha}

\newcommand{\bnu}{\boldsymbol \nu}
\newcommand{\blambda}{\boldsymbol \lambda}

\newcommand{\sch}{\mathrm{Sch}}
\newcommand{\sets}{\mathrm{Sets}}

\newtheorem{theorem}{Theorem}[section]
\newtheorem{lemma}{Lemma}[section]
\newtheorem{remark}{Remark}[section]
\newtheorem{corollary}{Corollary}[section]

\newtheorem{proposition}{Proposition}[section]

\newtheorem{definition}{Definition}[section]

\begin{document}

\title[Moduli space of ramified connections]
{Moduli space of irregular singular parabolic connections of generic ramified type
on a smooth projective curve}
\author{Michi-aki Inaba}
\address{Michi-aki Inaba \\
Department of Mathematics, Kyoto University, Kyoto, 606-8502, Japan}
\email{inaba@math.kyoto-u.ac.jp}
\subjclass[2010]{14D20, 53D30, 32G34, 34M55}
\maketitle

\begin{abstract}
We give an algebraic construction of
the moduli space of irregular singular connections
of generic ramified type on a smooth projective curve.
We prove that the moduli space is smooth and give its dimension.
Under the assumption that the exponent of ramified type is generic,
we give an algebraic symplectic form on the moduli space.
\end{abstract}

\section*{Introduction}

Let $C$ be a smooth projective curve and $E$ be an algebraic vector bundle on $C$.
Consider an algebraic connection
$\nabla\colon E\longrightarrow E\otimes\Omega^1_C(D)$
admitting poles along an effective divisor $D$ on $C$.
$\nabla$ is said to be regular singular at $t\in D$ if
$\nabla$ has a simple pole at $t$.
$\nabla$ is said to be irregular singular at $t\in D$
if the pole order of $\nabla$ at $t$ is greater than one.
We say that an irregular singular connection $\nabla$ is generic unramified at $t\in D$
if the leading term $\nabla|_t$ of the restriction
$\nabla|_{mt}\colon E|_{mt}\longrightarrow E|_{mt}\otimes\Omega_C(D)|_{mt}$
has the distinct eigenvalues, where $m$ is the pole order of $\nabla$ at $t$.
The generic unramified connections are most generic irregular singular connections.
The second generic irregular singular connections are generic ramified connections,
which is of the following type.
We say that an irregular singular connection
$(E,\nabla)$ of rank $r$ is generic $\nu$-ramified at $t\in D$ if
its completion $(\hat{E},\hat{\nabla})$ at $t$ is isomorphic to the connection
\[
 \nabla_{\nu}\colon \mathbb{C}[[w]] \ \ni \ f \ \mapsto \ df+\nu(w)f\ 
 \in \ \mathbb{C}[[w]]\otimes\frac{dz}{z^m}
\]
for $w=z^{\frac{1}{r}}$ and $\nu(w)\in\sum_{l=0}^{mr-r}\mathbb{C}w^ldw/w^{mr-r+1}$.
The aim of this paper is to construct the moduli space of generic ramified connections.

It is a classical result by R.~Fuchs that
Painlev\'e VI equations can be obtained as the
isomonodromic deformation of rank two regular singular connections
on $\mathbb{P}^1$ with four poles
and other types of Painlev\'e equations are known to be obtained
as generalized isomonodromic deformations of 
rank two connections on $\mathbb{P}^1$ with irregular singularities
(\cite {Jimbo-Miwa-Ueno}, \cite{Jimbo-Miwa-2}, \cite{Jimbo-Miwa-3}).
The space of initial conditions of Painlev\'e equations
are constructed by K.~Okamoto in \cite{Okamoto}
for all types.
Their compactifications are
classified by H.~Sakai in \cite{Sakai},
whose geometry characterizes the Painlev\'e equations.
If one wants to formulate the geometry of isomonodromic deformation
in a general framework,
an appropriate construction of the moduli space of connections is required.
Moduli space of regular singular connections are constructed by
N. Nitsure in \cite {Nitsure}, though the moduli space may have singularities.
Moduli space of regular singular connections with parabolic structure
becomes smooth with a symplectic structure,
which is constructed in the joint work \cite{IIS-1} with K. Iwasaki and M.-H. Saito
and in \cite{Inaba-1}.
In the case of rank two connections on $\mathbb{P}^1$
with $4$ regular singular points,
the moduli space of regular singular parabolic connections is isomorphic to
the space of initial conditions of Painlev\'e VI equations
constructed by K.~Okamoto.
Furthermore, we can construct in \cite {IIS-2} a compactified moduli space
which is isomorphic to the Sakai's rational surface in
\cite{Sakai} or in  \cite{Saito-Takebe-Terajima}.
We can obtain the geometric Painlev\'e property (see \cite{IIS-3} and \cite[Definition 2.4]{Inaba-1})
of the isomonodromic deformation on the family of moduli spaces
and then we can say that the constructed moduli space is
the space of initial conditions of the isomonodromic deformation.
Note that we cannot assume the underlying vector bundle trivial,
even when the base curve is $\mathbb{P}^1$ and
the degree of the underlying bundle is zero,
for the purpose of obtaining the geometric Painlev\'e property.

Moduli space of generic unramified connections is analytically constructed by
O.~Biquard and P.~Boalch
in \cite{Biquard-Boalch}.
In \cite{Boalch-1}, P. Boalch gives an algebraic construction of 
the moduli space of unramified irregular singular connections
over the trivial bundle on $\mathbb{P}^1$.
In a general case involving a higher genus curve, the joint work \cite{Inaba-Saito} 
with M.-H. Saito provides an algebraic construction 
of the moduli space of unramified irregular singular connections.
Again we do not  assume the triviality of the underlying bundle even if
the base curve is $\mathbb{P}^1$ and the degree of the bundle is zero.

Compared with the unramified case, the algebraic construction of the moduli space of
ramified irregular singular connections is difficult.
In \cite{Bremer-Sage}, C.~L.~Bremer and D.~S.~Sage
construct the algebraic moduli space of ramified connections over
a trivial bundle on $\mathbb{P}^1$.
They give the characterization of ramified connections
via an exhaustive consideration from
the representation theory.
The method of the construction of the moduli space is
similar to that in \cite{Boalch-1}.
P.~Boalch  constructed in \cite{Boalch-1} the moduli space of
generic unramified irregular singular connections 
which is locally framed at singular points and called it the extended moduli space.
The moduli space of irregular singular connections
is obtained as a symplectic reduction of the extended moduli space.
K.~Hiroe and D.~Yamakawa gave in \cite{Hiroe-Yamakawa}
a clear summary of the construction of the moduli space as a symplectic reduction
of the extended moduli space.
Here we remark that in all these results, the underlying vector bundles on $\mathbb{P}^1$
are assumed to be trivial.

For the character variety side, the moduli spaces were constructed in 
\cite{Boalch-3}, \cite{Boalch-4}, \cite{Boalch-Yamakawa} and \cite{Put-Saito},
which are expected to be related with appropriate moduli spaces of
irregular singular connections via the generalized Riemann-Hilbert correspondence.

In this paper, we construct a moduli space of generic ramified connections
on a smooth projective curve.
For the construction of the moduli space, it is important to give
an appropriate formulation of ramified connections.
In fact, it is necessary to rephrase the formal data
by the data on the restriction $(E,\nabla)|_{mt}$ to each pole divisor,
without depending on a framing.
First we consider a filtration
$E|_{mt}=V_0\supset V_1\supset\cdots\supset V_{r-1}\supset zV_0$
whose idea is similar to that in \cite{Bremer-Sage}.
The filtration is  given by
$\mathbb{C}[w]/(w^{mr})\supset (w)/(w^{mr})\supset\cdots\supset (w^{r-1})/(w^{mr})
\supset(w^r)/(w^{mr})$
when $(\hat{E},\hat{\nabla})=(\mathbb{C}[[w]],\nabla_{\nu})$.
Next we introduce quotient free $\mathbb{C}[w]/(w^{mr-r+1})$-module
$V_k\otimes\mathbb{C}[w]/(w^{mr-r+1})\xrightarrow{\pi_k}L_k$
of rank one for each $k$,
whose meaning is a quotient $\nu(w)$-eigenspace of
$\nabla|_{mt}\otimes\mathrm{id}$.
In order to recover a formal ramified connection,
we need not only the genericity condition on $\nu(w)$
but compatibility conditions between the
data $\{V_k,L_k,\pi_k\}$.
The precise definition is given in Definition \ref  {definition: ramified structure}
and Definition \ref {def:connection-generic-ramified}.
We should be careful that in our general setting, the true formal data
may not be the one given by the exponent if the exponent is not generic.
We adopt this general setting, because of the author's hope to get
a moduli space as a canonical degeneration of the moduli space of
regular singular or unramified irregular singular connections.

Once the formulation of ramified connection is established,
the construction of the moduli space becomes a standard task. 
For the construction of the moduli space
in Theorem \ref {thm-existence-moduli-generic},
we use a locally closed embedding of the moduli space
to the moduli space of parabolic $\Lambda^1_D$-triples constructed in
\cite{IIS-1}. 
The basic idea of the moduli space construction was inspired by Simpson's works
in \cite{Simpson-1} and \cite{Simpson-2},
which is similar to the GIT construction of the moduli space of vector bundles
using  the Quot-scheme.
The smoothness and the dimension counting of the moduli space
in Theorem \ref  {thm:generic-smooth} follow from
a standard deformation theory, using the idea of the methods in
\cite{Inaba-1} and \cite{Inaba-Saito}.
In Theorem \ref  {thm:symplectic-generic-ramified},
we construct a symplectic form on the moduli space of ramified connections.
For the proof of the $d$-closedness of the symplectic form,
we construct a family of moduli spaces of connections,
whose special fiber is the moduli space of ramified connections
and whose generic fiber is the moduli space of regular singular connections.
The $d$-closedness of the symplectic form on the moduli space of ramified connections
follows from that on the moduli space of regular singular connections.
A further consideration on the generic unramified irregular singular locus
provides a non-degeneracy of the symplectic form.

In the earlier version of the preprint, there was a serious error
in the formulation of ramified connection.
The author missed the condition (v) of
Definition \ref {definition: ramified structure}
or the condition (d) of Definition \ref  {def:moduli-functor}.
An example concerning this point is given in the appendix.

\vspace{10pt}

\noindent
{\bf Acknowledgements.}
This work is partially supported by
JSPS Grant-in-Aid for Scientific Research (C) 26400043
and (C) 19K03422.

\section{Formal data of a ramified connection and its paraphrase}

In this section we first recall elementary properties of formal connections of ramified type.

Consider the formal power series ring $\mathbb{C}[[z]]$
and denote by $\mathbb{C}((z))$ its quotient field.
We say that $(W,\nabla)$ is a formal connection over $\mathbb{C}[[z]]$ if
$W$ is a free $\mathbb{C}[[z]]$-module of finite rank and
$\nabla:W \longrightarrow W\otimes\mathbb{C}[[z]]\cdot\dfrac{dz}{z^m}$
is a $\mathbb{C}$-linear map satisfying $\nabla(fa)=a\otimes df+f\nabla(a)$
for $f\in\mathbb{C}[[z]]$ and $a\in W$.
Let $w$ be a variable with $w^l=z$ for a positive integer $l$.
Then the formal power series ring $\mathbb{C}[[w]]$ becomes a free
$\mathbb{C}[[z]]$-module of rank $l$.
Throughout this paper, we drop the subscript in the tensor product
$W\otimes_{\mathbb{C}[[z]]}\mathbb{C}[[w]]$
and simply write
$W\otimes\mathbb{C}[[w]]$.

The following is a fundamental classification theorem
of formal irregular singular connections.

\begin{theorem}[Hukuhara-Turrittin Theorem
(\cite{Babbitt-Varadarajan}, Proposition 1.4.1 or \cite{Sibuya}, Theorem 6.8.1)]
\label {Hukuhara-Turrittin Theorem}
For a formal irregular singular connection
$(W,\nabla)$ over $\mathbb{C}[[z]]$,
there is a positive integer $l$ and for a variable $w$ with $w^l=z$,
there are
$\nu_1(w),\ldots,\nu_s(w)\in\sum_{k=0}^{mr-r}\mathbb{C}w^kdw/w^{ml-l+1}$,
positive integers $r_1,\ldots,r_s$
such that
\[
 (W\otimes\mathbb{C}((w)),\nabla\otimes\mathbb{C}((w)))
 \cong
 (\mathbb{C}((w))^{r_1},d+J(\nu_1(w),r_1))\oplus\cdots
 \oplus(\mathbb{C}((w))^{r_s},d+J(\nu_s(w),r_s)),
\]
where $d+J(\nu_i(w),r_i):\mathbb{C}((w))^{r_i}\longrightarrow\mathbb{C}((w))^{r_i}\dfrac{dz}{z^m}$
is the connection given by
\[
 \begin{pmatrix} f_1 \\ \vdots \\ f_{r_i} \end{pmatrix}
 \mapsto
 \begin{pmatrix} df_1 \\ \vdots \\ df_{r_i} \end{pmatrix}
 +
 \left(
 \begin{pmatrix}
  \nu_i(w) & 0 & \cdots & 0 \\
  0 & \nu_i(w)  & \cdots & 0 \\
  \vdots & \vdots & \ddots & \vdots \\
  0 & 0 & \cdots & \nu_i(w)
 \end{pmatrix}
 +
 \dfrac{dw}{w}
 \begin{pmatrix}
  0 & 1 & 0 & \cdots & 0\\
  0 & 0 & 1 & \cdots & 0\\
  \vdots & \vdots &  \vdots& \ddots & \vdots \\
  0 & 0 & 0 & \cdots & 1 \\
  0 & 0  & 0 & \cdots & 0
 \end{pmatrix}
 \right) 
 \begin{pmatrix} f_1 \\ \vdots \\ f_{r_i} \end{pmatrix}.
\]
Here $\nu_i(w)-\nu_j(w)\notin\mathbb{Z}\dfrac{dw}{w}$ unless $\nu_i(w)=\nu_j(w)$.
\end{theorem}

We note that
$\nu_1(w),\ldots,\nu_s(w) \mod \mathbb{Z}\dfrac{dw}{w}$
are invariants of $(W,\nabla)$.
We say that $\nu_1(w),\ldots,\nu_s(w)$ are generalized eigenvalues of $\nabla$
or exponents of $\nabla$.
In general, a ramified cover $w\mapsto w^l=z$ is necessary for
obtaining the invariants $\nu_1(w),\ldots,\nu_s(w)$.
If we can take all $\nu_1(w),\ldots,\nu_s(w)$
as differential forms in $z$,
we say that $\nabla$ is unramified.

\begin{definition}\rm
We say that a formal connection $(V,\nabla)$ over $\mathbb{C}[[z]]$ is
formally irreducible if there is no subbundle
$0\neq W \subsetneq V$ satisfying
$\nabla(W) \subset W \otimes\dfrac{dz}{z^m}$.
\end{definition}

We can see the following proposition immediately by induction on $\rank W$.

\begin{proposition}
For a formal connection $(W,\nabla)$ over $\mathbb{C}[[z]]$
with $\rank_{\mathbb{C}[[z]]}W<\infty$,
there is a filtration
$0=W_s\subset W_{s-1}\subset \cdots\subset W_0=W$
such that $\nabla(W_i)\subset W_i\otimes\dfrac{dz}{z^m}$,
each $W_i/W_{i+1}$ is a free $\mathbb{C}[[z]]$-module
and $(W_i/W_{i+1},\nabla_{W_i/W_{i+1}})$ is formally irreducible,
where $\nabla_{W_i/W_{i+1}}$ is the connection on $W_i/W_{i+1}$
induced by $\nabla|_{W_i}$.
\end{proposition}

We will see what data determines the formal connection
$(W_i/W_{i+1},\nabla_{W_i/W_{i+1}})$.
So we assume that $(W,\nabla)$ is formally irreducible for simplicity.
Thanks to the Hukuhara-Turrittin Theorem (Theorem \ref {Hukuhara-Turrittin Theorem}),
there is a ramified cover $w\mapsto w^l=z$ and
a generalized eigenvalue $\nu(w)\in\mathbb{C}[[w]]\dfrac{dz}{z^m}$ of
$(W\otimes\mathbb{C}((w)),\nabla\otimes\mathbb{C}((w)))$.
In other words, there is an eigen vector $v\in W\otimes\mathbb{C}((w))$,
that is, $\nabla (v)=\nu(w) v$.
If we take a generator $\sigma$ of $\mathrm{Gal}(\mathbb{C}((w))/\mathbb{C}((z)))$,
then the $\mathbb{C}((w))$-subspace generated by
$\{\sigma^k(v)|0\leq k\leq l-1\}$
descends to a subspace $W'$ of $W\otimes\mathbb{C}((z))$.
If we put $\tilde{W'}:=W\cap W'$, then  $\tilde{W'}$ is a subbundle of $W$
preserved by $\nabla$ and $\tilde{W'}\otimes\mathbb{C}((z))=W'$.
Since $(W,\nabla)$ is irreducible,
we have $\tilde{W'}=W$.
If $j$ is the minimum among
$0<j\leq l$ satisfying $\sigma^j(\nu(w))=\nu(w)$,
then $j|l$ and $\nu(w)\in\mathbb{C}[[w^{\frac{l}{j}}]]\dfrac{dz}{z^m}$.
So we may replace $w$ by $w^{\frac{l}{j}}(=z^{\frac{1}{j}})$
and then  $\{\sigma^k(\nu(w))|0\leq k\leq l-1\}$
are distinct.
We have $l=\rank W$ in this case. 
For an eigenvector $v\in W\otimes\mathbb{C}((w))$, we have
\[
 \nabla(\sigma^k(v))=\sigma^k(\nabla(v))=\sigma^k(\nu(w)v)=\sigma^k(\nu(w))\sigma^k(v).
\]
So $v,\sigma(v),\ldots,\sigma^{l-1}(v)$ are linearly independent over $\mathbb{C}((w))$,
because they are eigenvectors of $\nabla$ with the distinct eigenvalues
$\nu(w),\sigma(\nu(w)),\ldots,\sigma^{l-1}(\nu(w))$.
So we can see that $W\otimes\mathbb{C}((w))$ has  $\sigma^k(\nu(w))$-eigensubspaces,
but we can also regard that $W\otimes\mathbb{C}((w))$ has
quotient $\sigma^k(\nu(w))$-eigenspaces.
Replacing $\nu(w)$ by adding an element of $\mathbb{Z}\dfrac{dw}{w}$,
we may assume that
there is a surjection
$\pi\colon W\otimes\mathbb{C}[[w]]\longrightarrow\mathbb{C}[[w]]$
such that the diagram
\[
 \begin{CD}
  W\otimes\mathbb{C}[[w]] @>\pi>> \mathbb{C}[[w]] \\
  @V\nabla\otimes\mathrm{id} VV  @VV\nabla_{\nu}V  \\
  W\otimes\mathbb{C}[[w]]\cdot\dfrac{dz}{z^m}
  @>\pi\otimes\mathrm{id}>> \mathbb{C}[[w]]\dfrac{dz}{z^m}
 \end{CD}
\]
is commutative,
where $\nabla_{\nu}$ is given by $\nabla_{\nu}(f(w))=df(w)+f(w)\nu(w) $ 
for $f(w)\in\mathbb{C}[[w]]$.
Note that $W\xrightarrow{\pi|_W}\mathbb{C}[[w]]$ is injective,
because $(W,\nabla)$ is formally irreducible.
So $\coker(\pi|_W)$ is of finite length,
since $\rank_{\mathbb{C}[[z]]}W=\rank_{\mathbb{C}[[z]]}\mathbb{C}[[w]]=l$.

Conversely, assume that
$\nu(w)\in\sum_{s=0}^{mr-r}\mathbb{C}w^s dw/w^{mr-r+1}$
is given for $w=z^{\frac{1}{r}}$ such that
$\tau\nu(w)\neq\nu(w)$ for any $\tau\in\mathrm{Gal}(\mathbb{C}((w))/\mathbb{C}((z)))$
other than $\mathrm{id}$.
If $\mathbb{C}[[w]]\xrightarrow{q}Q$ is a quotient $\mathbb{C}[[z]]$-module
of finite length satisfying the commutative diagram
\[
 \begin{CD}
  \mathbb{C}[[w]] @>q>> Q \\
  @V\nabla_{\nu}VV   @VV\nabla_{Q}V \\
  \mathbb{C}[[w]]\otimes\dfrac{dz}{z^m} @>q\otimes\mathrm{id}>> 
  Q\otimes\dfrac{dz}{z^m},
 \end{CD}
\]
for some morphism $\nabla_{Q}$,
then a formal connection
$\nabla|_{\ker q}\colon \ker q\longrightarrow \ker q\otimes dz/z^m$
is induced.
If $0\neq U\subset \ker q$ is a subbudle preserved by
$\nabla|_{\ker q}$,
then we have
\[
 (U,\nabla|_U)\otimes_{\mathbb{C}[[z]]}\mathbb{C}((w))
 \subset
 (\ker q,\nabla|_{\ker q})\otimes_{\mathbb{C}[[z]]}\mathbb{C}((w))
 \cong
 \bigoplus_{k=0}^{r-1} \big( \mathbb{C}((w)),\nabla_{\sigma^k(\nu(w))} \big),
\]
where $\sigma$ is a generator of
$\mathrm{Gal}(\mathbb{C}((w))/\mathbb{C}((z)))$
and $\big( \mathbb{C}((w)),\nabla_{\sigma^k(\nu(w))} \big)$
means the $\sigma^k(\nu(w))$-eigenspace.
In particular, we can choose a vector
$0\neq u\in U\otimes\mathbb{C}((w))$
satisfying $\nabla(u)=\sigma^k(\nu(w))u$
for some $k$.
By the assumption, $\nu(w),\sigma(\nu(w)),\ldots,\sigma^{r-1}(\nu(w))$
are mutually distinct.
So 
$u,\sigma(u),\ldots,\sigma^{r-1}(u)$
are linearly independent over $\mathbb{C}((w))$,
because they are the eigenvectors with distinct eigenvalues
$\sigma^k(\nu(w)),\sigma^{k+1}(\nu(w)),\ldots,\sigma^{r-1+k}(\nu(w))$.
Thus $\rank_{\mathbb{C}[[z]]} U=\dim_{\mathbb{C}((w))}(U\otimes\mathbb{C}((w)))=r$
which implies $U=\ker q$.
Hence $(\ker q,\nabla|_{\ker q})$ is a formally irreducible connection.

Summarizing the above arguments, we obtain the following proposition:

\begin{proposition}
\label {proposition: characterization of ramified connection}
A formal connection $(W,\nabla)$ of rank $r$ over $\mathbb{C}[[z]]$
is irreducible if and only if it is isomorphic to $(\ker q,\nabla_{\nu}|_{\ker q})$,
where
$\mathbb{C}[[w]]\stackrel{q}\longrightarrow Q$ is a quotient $\mathbb{C}[[z]]$-module
of finite length for  $w=z^{\frac{1}{r}}$
and 
$\nu(w)\in \sum_{k=0}^{mr-r}\mathbb{C}w^kdw/w^{mr-r+1}$
is not fixed by the elements of $\mathrm{Gal}(\mathbb{C}((w))/\mathbb{C}((z)))$
other than $\mathrm{id}$,
such that the diagram
\[
 \begin{CD}
 \mathbb{C}[[w]] @>q>> Q \\
 @V \nabla_{\nu} VV @VV \nabla_Q V \\
 \mathbb{C}[[w]]\otimes\dfrac{dz}{z^m}  @>q\otimes\mathrm{id}>> Q\otimes\dfrac{dz}{z^m} 
 \end{CD}
\] 
is commutative
for a morphism $\nabla_Q\colon Q\longrightarrow Q\otimes dz/z^m$
satisfying $Q(f(z)a)=a\otimes df(z)+f(z)a$ ($a\in Q$, $f(z)\in\mathbb{C}[[z]]$).
\end{proposition}

\begin{remark} \label {remark: formal quotient}
\rm
Let $(W,\nabla)$ be a formal irreducible connection over $\mathbb{C}[[z]]$
of rank $r$
and take a variable $w$ with $w^r=z$.
Then we can see from the proof of 
Proposition \ref {proposition: characterization of ramified connection}
that there is a surjection
$\pi\colon W\otimes\mathbb{C}[[w]]\longrightarrow \mathbb{C}[[w]]$
whose restriction $\pi|_W\colon W\longrightarrow \mathbb{C}[[w]]$
is generically injective with finite length cokernel,
such that the diagram
\[
 \begin{CD}
  W\otimes\mathbb{C}[[w]] @>\pi >> \mathbb{C}[[w]]  \\
  @V \nabla\otimes\mathrm{id} VV   @VV \nabla_{\nu} V \\
  W\otimes\mathbb{C}[[w]]\otimes\dfrac{dz}{z^m}@>\pi\otimes \mathrm{id}>> 
  \mathbb{C}[[w]]\otimes \dfrac{dz}{z^m}
 \end{CD}
\]
is commutative.
If $\sigma$ is a generator of the Galois group
$\mathrm{Gal}(\mathbb{C}((w))/\mathbb{C}((z)))$,
\[
 \sigma^{k}\circ\pi\circ(1\otimes\sigma)^{-k}\colon
 W\otimes\mathbb{C}[[w]]\longrightarrow\mathbb{C}[[w]]
\]
is a surjective homomorphism of $\mathbb{C}[[w]]$-modules,
which makes the diagram
\[
 \begin{CD}
  W\otimes\mathbb{C}[[w]] @> \sigma^{k}\circ\pi\circ(1\otimes\sigma)^{-k} >> \mathbb{C}[[w]]  \\
  @V \nabla\otimes\mathrm{id} VV   @VV \nabla_{\sigma^k(\nu)} V \\
  W\otimes\mathbb{C}[[w]]\otimes\dfrac{dz}{z^m}
  @>(\sigma^{k}\circ\pi\circ(1\otimes\sigma)^{-k})\otimes \mathrm{id}>> 
  \mathbb{C}[[w]]\otimes \dfrac{dz}{z^m}
 \end{CD}
\]
commutative for $0\leq k\leq r-1$.
\end{remark}

\begin{lemma}
\label {lem:generic-1}
Let $(W,\nabla)$ be a formal irreducible connection of rank $r$
over $\mathbb{C}[[z]]$ with a quotient
$\pi\colon W\otimes\mathbb{C}[[w]]\longrightarrow\mathbb{C}[[w]]$
for $w^r=z$ satisfying the commutative diagram
\[
 \begin{CD}
 W\otimes\mathbb{C}[[w]] @>\pi>> \mathbb{C}[[w]] \\
 @V\nabla\otimes\mathrm{id}VV  @VV\nabla_{\nu} V \\
 W\otimes\mathbb{C}[[w]]\otimes\dfrac{dz}{z^m} @>\pi\otimes\mathrm{id}>>
 \mathbb{C}[[w]]\otimes\dfrac{dz}{z^m},
 \end{CD}
\]
where $\nu(w)\in\sum_{l=0}^{mr-r}\mathbb{C}w^kdw/w^{mr-r+1}$.
Assume that 
the $\dfrac{wdw}{w^{mr-r+1}}$-coefficient
of $\nu(w)$ is non-zero.
Then the restriction
\[
 \pi|_W\colon W\xrightarrow{\sim}\mathbb{C}[[w]]
\] 
is in an isomorphism.
\end{lemma}

\begin{proof}
We can write
\[
 \nu(w)=\nu_0(z)+\nu_1(z)w+\cdots+\nu_{r-1}(z)w^{r-1}
\]
with $\nu_k(z)\in\mathbb{C}[z]dz/z^m$.
By the assumption,
$\nu_1(z)$ is a generator of $\mathbb{C}[[z]]dz/z^m$.
So we can write
$\nu_1(z)=(c_0+c_1z+\cdots+c_{m-2}z^{m-2})dz/z^m$
with $c_0\neq 0$.
Consider the connection
\[
 \nabla-\nu_0(z)\mathrm{id}\colon \
 W \ \ni \ v \ \mapsto \ \nabla(v)-\nu_0(z)v
 \ \in \ W\otimes\frac{dz}{z^m}.
\]
Then we have the commutative diagram
\[
 \begin{CD}
 W\otimes\mathbb{C}[[w]] @>\pi>> \mathbb{C}[[w]] \\
 @V(\nabla-\nu_0(z)\mathrm{id})\otimes\mathrm{id}VV  @VV\nabla_{\nu(w)-\nu_0(z)} V \\
 W\otimes\mathbb{C}[[w]]\otimes\dfrac{dz}{z^m} @>\pi\otimes\mathrm{id}>>
 \mathbb{C}[[w]]\otimes\dfrac{dz}{z^m}.
 \end{CD}
\]
In particular, $\pi(W)\subset\mathbb{C}[[w]]$ is preserved by
$\nabla_{\nu(w)-\nu_0(z)}$.
Since $\pi$ is surjective, there exists $v\in W$
satisfying
$\pi(v)=1+wa_0(w)$ for some $a_0(w)\in\mathbb{C}[[w]]$.
We can write
$\nu(w)-\nu_0(z)=c_0w(1+wb_1(w))dw/w^{mr-r+1}$
for some $b_1(w)\in\mathbb{C}[[w]]$.
So we can write
\begin{align*}
 \pi((\nabla-\nu_0(z)\mathrm{id})(v))
 &=
 (\nabla_{\nu(w)-\nu_0(z)})(\pi(v))
 \\
 &=
 \left( c_0w+(c_0b_1(w)+c_0a_0(w))w^2+(\text{higher order terms})\right)dw/w^{mr-r+1}.
\end{align*}
In particular, $\pi(V)$ contains an element of the form
$w(1+wa_1(w))$ for some $a_1(w)\in\mathbb{C}[[w]]$.
Repeating this procedure, we can construct elements
\[
 1+wa_0(w),w(1+wa_1(w)),\ldots,w^{r-1}(1+wa_{r-1}(w))
\]
of $\pi(W)$.
Since all the elements of $\mathbb{C}[[w]]$ can be written as
a $\mathbb{C}[[z]]$-linear combination of the above elements,
we get the surjectivity of $\pi|_W\colon W\longrightarrow \mathbb{C}[[w]]$.
\end{proof}

\begin{remark} \rm
In this paper we do not treat formal connections with $\coker(\pi|_W)=Q$ non-trivial.
We mainly treat connections satisfying $\coker(\pi|_W)=0$ and we will say such
type of connection a generic ramified type.
A precise definition of connections of generic ramified type is given in
Definition \ref{def:connection-generic-ramified},
while the essential part of the definition is given in
Definition \ref {definition: ramified structure}.
\end{remark}

Consider the formal connection
$(W,\nabla)\cong(\mathbb{C}[[w]],\nabla_{\nu})$.
Then we get a filtration
\[
 W=\hat{V}_0\supset \hat{V}_1\supset\cdots\supset \hat{V}_{r-1}\supset \hat{V}_r=z\hat{V}_0
\]
by pulling back the filtration
$\mathbb{C}[[w]]\supset (w)\supset(w^2)\supset\cdots\supset(w^{r-1})\supset (w^r)$.
There is a canonical surjection
$\pi\colon W\otimes_{\mathbb{C}[[z]]}\mathbb{C}[[w]]\longrightarrow\mathbb{C}[[w]]$
which induces commutative diagrams
\[
 \begin{CD}
 \hat{V}_k\otimes_{\mathbb{C}[[z]]}\mathbb{C}[[w]] 
 @>\pi|_{\hat{V}_k}\otimes\mathrm{id}_{\mathbb{C}[[w]]} >> (w^k) \\
 @V\nabla|_{\hat{V}_k}\otimes\mathrm{id}VV  @VV\nabla_{\nu(w)}|_{(w^k)} V \\
 \hat{V}_k\otimes_{\mathbb{C}[[z]]}\mathbb{C}[[w]]\otimes\dfrac{dz}{z^m} 
 @>\pi|_{\hat{V}_k}\otimes\mathrm{id}_{\mathbb{C}[[w]]} >>
 (w^k)\otimes\dfrac{dz}{z^m}
 \end{CD}
\]
for $0\leq k\leq r-1$.
Note that the restriction
$\nabla_{\nu}|_{(w^k)}$ is given by
\[
 \nabla_{\nu(w)} \left( w^kf(w) \right) 
 =
 d(w^k)f(w)+w^kdf(w)+w^kf(w)\nu(w)
 =w^k \left( \nabla_{\nu(w)}+\frac{k}{r}\frac{dz}{z}\mathrm{id} \right)(f(w)).
\]
Let us consider the restriction of the above data to the finite subscheme
defined by $z^m=0$.
So the filtration $\hat{V}_k$ induces a filtration
\[
 W\otimes_{\mathbb{C}[[z]]}\mathbb{C}[z]/(z^m)=V_0\supset V_1\supset\cdots\supset V_{r-1}
 \supset V_r=zV_0.
\]
Define
\begin{align*}
 \bar{\pi}_k
 &:=
 (\pi|_{\hat{V}_k\otimes\mathbb{C}[[w]] })\otimes\mathrm{id}_{\mathbb{C}[w]/(w^{mr-r+1})}
 \colon
 V_k\otimes_{\mathbb{C}[z]/(z^m)}\mathbb{C}[w]/(w^{mr-r+1})\longrightarrow (w^k)/(w^{k+mr-r+1})
 \\
 \bar{\nabla}_k
 &:=
 \nabla|_{z^m=0} \colon
 V_k \longrightarrow V_k\otimes\frac{dz}{z^m}.
\end{align*}
Then there are commutative diagrams
\[
 \begin{CD}
  V_k\otimes_{\mathbb{C}[z]/(z^m)}\mathbb{C}[w]/(w^{mr-r+1}) 
  @>\bar{\pi}_k>> (w^k)/(w^{k+mr-r+1}) \\
  @V\bar{\nabla}_k\otimes\mathrm{id} VV  @VV\nabla_{\nu(w)+\frac{k}{r}\frac{dz}{z}}V \\
  V_k\otimes\dfrac{dz}{z^m}\otimes_{\mathbb{C}[z]/(z^m)}\mathbb{C}[w]/(w^{mr-r+1}) 
  @>\bar{\pi}_k\otimes \mathrm{id} >> (w^k)/(w^{k+mr-r+1})\otimes\dfrac{dz}{z^m} 
 \end{CD}
\]
for $0\leq k\leq r-1$.

There are compatibility conditions among
$V_k,\pi_k$ for $k=0,1\ldots,r-1$.
Later in Proposition \ref  {prop:generic-2},
we will prove that the compatibility conditions can recover a formal
ramified connection under a genericity condition
on the exponent $\nu(w)$.
We summarize the conditions in the following definition.
The advantage of the definition is to be fit for the
formulation of global moduli of ramified connections.

\begin{definition}
\label {definition: ramified structure}
\rm
Let $w$ be a variable with $w^r=z$ and take
$\nu(w)\in\sum_{l=0}^{mr-r}\mathbb{C}w^ldw/w^{mr-r+1}$.
Assume that
$(\overline{W},\overline{\nabla})$ is a pair of a free $\mathbb{C}[z]/(z^m)$-module
$\overline{W}$ of rank $r$
and a homomorphism
$\overline{\nabla}\colon \overline{W} \longrightarrow \overline{W}\otimes dz/z^m$
of $\mathbb{C}[z]/(z^m)$-modules.
We say that 
$\big( (V_k,L_k,\pi_k)_{0\leq k\leq r-1},(\phi_k)_{1\leq k\leq r} \big)$
is a $\nu(w)$-ramified structure 
on $(\overline{W},\overline{\nabla})$ if
\begin{enumerate}
\item[(i)]
$\overline{W}
=V_0\supset V_1\supset\cdots\supset V_{r-1}\supset zV_0=:V_r$
is a filtration satisfying
$\overline{\nabla}(V_k)\subset V_k\otimes dz/z^m$
and $\length(V_k/V_{k+1})=1$ for $0\leq k\leq r-1$,
\item[(ii)]
$\pi_k\colon V_k\otimes\mathbb{C}[w]/(w^{mr-r+1})\longrightarrow L_k$
is a quotient free $\mathbb{C}[w]/(w^{mr-r+1})$-module of rank one
such that the restriction
$\pi_k|_{V_k}\colon V_k\longrightarrow L_k$ is surjective
for $0\leq k\leq r-1$,
\item[(iii)]
the diagrams
\[
 \begin{CD}
 V_k\otimes\mathbb{C}[w]/(w^{mr-r+1})  
 @> \pi_k  >> L_k  \\
 @V\overline{\nabla}|_{V_k}\otimes\mathrm{id}  VV   @VV \nu(w)+\frac{k}{r}\frac{dz}{z} V \\
 V_k\otimes \dfrac{dz}{z^m}\otimes\mathbb{C}[w]/(w^{mr-r+1})  
 @>\pi_k\otimes\mathrm{id} >> L_k\otimes\dfrac{dz}{z^m}
 \end{CD}
\]
are commutative for $0\leq k\leq r-1$,
\item[(iv)]
$\phi_k\colon L_k\longrightarrow L_{k-1}$
are homomorphisms of $\mathbb{C}[w]$-modules
with factorizations
\[
 \phi_k \colon
  L_k\xrightarrow [\sim] {\psi_k} (w)/(w^{mr-r+2})\otimes L_{k-1}
  \longrightarrow wL_{k-1}\hookrightarrow L_{k-1}
\]
for isomorphisms
$\psi_k \colon L_k\xrightarrow{\sim}
(w)/(w^{mr-r+2})\otimes L_{k-1}$
of $\mathbb{C}[w]$-modules
for $1\leq k\leq r-1$
and $\phi_r\colon (w^r)/(w^{mr+1})\otimes L_0\longrightarrow L_{r-1}$
is a $\mathbb{C}[w]$-homomorphism whose image is $wL_{r-1}$
such that
the diagram
\[
 \begin{CD}
  (z)/(z^{m+1})\otimes V_0\otimes\mathbb{C}[w]/(w^{mr-r+1}) 
  @> \mathrm{id}\otimes\pi_0  >> (w^r)/(w^{mr+1})\otimes L_0  \\
  @V VV  @V \phi_r VV \\
  V_{r-1}\otimes\mathbb{C}[w]/(w^{mr-r+1})     @> \pi_{r-1}   >> L_{r-1}
 \end{CD}
\]
is commutative and the diagrams
\[
 \begin{CD}
  V_k\otimes\mathbb{C}[w]/(w^{mr-r+1})   @>\pi_k|_{V_k} >>  L_k  \\
  @VVV  @VV  \phi_k V \\
  V_{k-1}\otimes\mathbb{C}[w]/(w^{mr-r+1})   @>\pi_{k-1}|_{V_{k-1}} >> L_{k-1}
 \end{CD}
\]
are commutative for $1\leq k\leq r-1$,
\item[(v)]
the composition
\begin{align} \label {equation: key coincidence with composition}
\begin{split}
 (w^r)\otimes L_0 
 \xrightarrow{\phi_r} L_{r-1}
 \xrightarrow[\sim]{\psi_{r-1}}
 (w)\otimes L_{r-2}
 &\xrightarrow[\sim]{\mathrm{id}_{(w)}\otimes\psi_{r-2}}
 (w^2)\otimes L_{r-3}
 \\
 &\xrightarrow[\sim]{\mathrm{id}_{(w^2)}\otimes\psi_{r-3}}\cdots
 \xrightarrow[\sim]{\mathrm{id}_{(w^{r-2})}\otimes\psi_1}
 (w^{r-1})\otimes L_0
\end{split}
\end{align}
coincides with the homomorphism
canonically induced by
$(w^r)\longrightarrow (w^{r-1})$.
\end{enumerate}
\end{definition}

\begin {remark}\rm
The condition of Definition \ref {definition: ramified structure} (v)
is independent of the choice of the lifts $\psi_k$ of $\phi_k$ chosen in (iv).
Indeed, the composition (\ref  {equation: key coincidence with composition})
in Definition \ref {definition: ramified structure} (v)
is determined only by $(\phi_k)_{1\leq k\leq r}$.
\end{remark}

\begin{proposition} \label {prop:generic-2} \rm
Assume that
$(W,\nabla)$ is a formal connection of rank $r$ over $\mathbb{C}[[z]]$
such that $(W,\nabla)\otimes\mathbb{C}[z]/(z^m)$
has a $\nu(w)$-ramified structure defined
in Definition \ref {definition: ramified structure},
where $w$ is a variable with $w^r=z$.
Furthermore, assume that 
the $w\, dw/w^{mr-r+1}$-coefficient of  $\nu(w)$
does not vanish.
Then
there is a formal isomorphism
$(W,\nabla)
\cong (\mathbb{C}[[w]],\nabla_{\nu(w)})$.
\end{proposition}

\begin{proof}
We write
$\nu(w)=(a_0(z)+a_1(z)w+\cdots+a_{r-1}(z)w^{r-1})dz/z^m$
for polynomials $a_0(z),\ldots,a_{r-1}(z)$ in $z$.
By the assumption, $a_0(z)$ is of degree at most $m-1$,
the degrees of $a_1(z),\ldots,a_{r-1}(z)$ are at most $m-2$
and the constant term $a_1(0)$ of $a_1(z)$ does not vanish.
Let
\[
 N'\colon W\otimes \mathbb{C}[z]/(z^{m-1})
 \longrightarrow  W\otimes \mathbb{C}[z]/(z^{m-1})
\]
be the endomorphism of $\mathbb{C}[z]/(z^{m-1})$-module
determined by
\[
 N'\dfrac{dz}{z^m}
 =
 \nabla\otimes\mathrm{id}_{\mathbb{C}[z]/(z^{m-1})}
 -\nu_0(z)\mathrm{id}_{W\otimes\mathbb{C}[z]/(z^{m-1})}
 \colon W\otimes\mathbb{C}[z]/(z^{m-1})
 \longrightarrow
 W\otimes\mathbb{C}[z]/(z^{m-1})\otimes\frac{dz}{z^m}.
\]
If we put $b(w):=a_1(z)w+a_2(z)w^2+\cdots+a_{r-1}(z)w^{r-1}$,
then the diagram
\[
 \begin{CD}
  W\otimes\mathbb{C}[z]/(z^{m-1}) 
  @>\pi_0|_{\bar{W}}\otimes\mathbb{C}[z]/(z^{m-1})>\cong> 
  L_0/w^{mr-r}L_0\cong\mathbb{C}[w]/(w^{mr-r}) \\
  @V N' VV  @VV b(w) V \\
  W\otimes\mathbb{C}[z]/(z^{m-1}) 
  @>\pi_0|_{\bar{W}}\otimes\mathbb{C}[z]/(z^{m-1})>\cong> 
  L_0/w^{mr-r}L_0\cong\mathbb{C}[w]/(w^{mr-r})
 \end{CD}
\]
is commutative.
Note that
$p_0:=\pi|_{\overline{W}}\otimes\mathbb{C}[z]/(z^{m-1})
\colon W\otimes\mathbb{C}[z]/(z^{m-1}) \longrightarrow
L_0/w^{mr-r}L_0$
is isomorphic, because
$\pi_0|_{\overline{W}}\colon\overline{W}=V_0\longrightarrow L_0\cong\mathbb{C}[w]/(w^{mr-r+1})$
is surjective by the assumption (ii) of
Definition \ref {definition: ramified structure}
and 
$\length(W\otimes\mathbb{C}[z]/(z^{m-1}))=r(m-1)=
\length(L_0/w^{mr-r}L_0)$. 
Using the condition $a_1(0)\neq 0$,
we can see that
\begin{align} \label {equation: generic nilpotent endomorphism}
\begin{split}
 \im (N')^{r-1} 
 &\not\subset z\big(W\otimes\mathbb{C}[z]/(z^{m-1})\big)
 \\
 \im (N')^r
 &\subset z\big( W\otimes\mathbb{C}[z]/(z^{m-1}) \big).
\end{split}
\end{align}

Assume that there is a subbundle
$0\neq U\subset W$ preserved by $\nabla$.
Then $U/z^{m-1}U\subset W/z^{m-1}W$ is a subbundle
preserved by $N'$.
In particular, $p_0(U/z^{m-1}U)\subset L_0/w^{mr-r}L_0$
is a $\mathbb{C}[z]/(z^{m-1})$-subbundle
preserved by $b(w)$.
So  $p_0(U/z^{m-1}U)$ contains an element of the form
$w^k(1+w\beta(w))$ with $0\leq k\leq r-1$.
Thus we can see that $p_0(U/z^{m-1}U)$ contains
$z\bar{W}$, since it is preserved by $b(w)$ satisfying $a_1(0)\neq 0$.
Since $U$ is a subbundle of $W$,
we have $U=W$ and we have proved that
$(W,\nabla)$ is a formal irreducible connection.

By Proposition \ref {proposition: characterization of ramified connection}
and Remark \ref  {remark: formal quotient},
there is 
$\nu'(w)\in\sum_{l=0}^{mr-r}\mathbb{C}w^ldw/w^{mr-r+1}$
and an injection
$W\hookrightarrow \mathbb{C}[[w]]$
such that the induced homomorphism
$\pi'\colon W\otimes\mathbb{C}[[w]]\longrightarrow \mathbb{C}[[w]]$
is surjective and that the diagram
\[
 \begin{CD}
  W\otimes\mathbb{C}[[w]] @>\pi' >> \mathbb{C}[[w]] \\
  @V \nabla\otimes\mathrm{id} VV  @VV \nabla_{\nu'(w)}V \\
  W\otimes\dfrac{dz}{z^m}\otimes\mathbb{C}[[w]] 
  @>\pi'\otimes\mathrm{id}>> \mathbb{C}[[w]]\otimes\dfrac{dz}{z^m}
 \end{CD}
\]
is commutative.
Write
$\nu'(w)=(a'_0(z)+a'_1(z)w+\cdots+a'_{r-1}(z)w^{r-1})dz/z^m$.
If we put
\[
 b'(w):=a'_0(z)-a_0(z)+a'_1(z)w+\cdots+a'_{r-1}(z)w^{r-1},
\]
then $\pi'$ induces a commutative diagram
\[
 \begin{CD}
  W\otimes\mathbb{C}[z]/(z^{m-1})  @> \pi'|_W\otimes\mathrm{id}_{\mathbb{C}[z]/(z^{m-1})} >>
  \mathbb{C}[w]/(w^{mr-r}) \\
  @V N'\otimes\mathrm{id} VV  @VV b'(w) V \\
  W\otimes\mathbb{C}[z]/(z^{m-1})  @> \pi'|_W\otimes\mathrm{id}_{\mathbb{C}[z]/(z^{m-1})} >>
  \mathbb{C}[w]/(w^{mr-r}).
 \end{CD}
\]
Since $\pi'\colon W\otimes\mathbb{C}[[w]]\longrightarrow \mathbb{C}[[w]]$ is surjective, 
$\im \big(\pi'|_W\otimes\mathrm{id}_{\mathbb{C}[z]/(z^{m-1})}\big)$
contains a generator of 
$\mathbb{C}[w]/(w^{mr-r})$
as a $\mathbb{C}[w]$-module.
Then, we can see from the property (\ref {equation: generic nilpotent endomorphism})
of $N'$ that
$b'(w)^{r-1}\notin z\mathbb{C}[w]/(w^{mr-r})$
and
$b'(w)^r\in z\mathbb{C}[w]/(w^{mr-r})$.
So we have $a'_0(z)\equiv a_0(z) \pmod{z}$
and $a'_1(0)\neq 0$.
Thus 
\[
 \pi'|_W\otimes\mathrm{id}_{\mathbb{C}[z]/(z^{m-1})}
 \colon W\otimes\mathbb{C}[z]/(z^{mr-r})
 \longrightarrow \mathbb{C}[w]/(w^{mr-r})
\]
is surjective and then it is isomorphic because
$\length (W\otimes\mathbb{C}[z]/(z^{m-1}))=\length(\mathbb{C}[w]/(w^{mr-r}))$.
Since the composition
\[
 \mathbb{C}[w]/(w^{mr-r}) 
 \xrightarrow {(\pi'|_W\otimes\mathrm{id}_{\mathbb{C}[z]/(z^{m-1})})\circ p_0^{-1}}
 \mathbb{C}[w]/(w^{mr-r})
\]
is an isomorphism of $\mathbb{C}[w]/(w^{mr-r})$-modules, we have
$b(w)\equiv b'(w) \pmod{w^{mr-r}}$,
 which means
$\nu(w)\equiv \nu'(w) \pmod{w^{mr-r}dw/w^{mr-r+1}}$.
In particular,
$\pi'|_W\colon W\longrightarrow\mathbb{C}[[w]]$ is an isomorphism
by Lemma \ref {lem:generic-1}.

We get a basis of $W\otimes\mathbb{C}[z]/(z^m)$
by pulling back the basis $1,w,\ldots,w^{r-1}$ of $\mathbb{C}[w]/(w^{mr})$
via the isomorphism $\pi'|_W\otimes\mathbb{C}[z]/(z^m)$
and the representation matrix of $\nabla\otimes\mathbb{C}[z]/(z^m)$
with respect to this basis is
\[
 \begin{pmatrix}
  \nu'_0(z) & z\nu'_{r-1}(z) & \cdots & z \nu'_1(z) \\
  \nu'_1(z) & \nu'_0(z)+\frac{dz}{rz} & \cdots & \nu'_2(z) \\
  \vdots & \vdots & \ddots & \ \vdots \\
  \nu'_{r-1}(z) & \nu'_{r-2}(z) & \cdots & \nu'_0(z)+\frac{(r-1)dz}{rz}
 \end{pmatrix}.
\]
In particular, we have
$\Tr(\nabla\otimes\mathbb{C}[z]/(z^m))=r\nu_0(z)+(r-1)dz/2z$.

Take a generator $\bar{e}_0$ of $L_0$ as a $\mathbb{C}[w]/(w^{mr-r+1})$-module.
Let $\bar{e}_k$ be the element of $L_k$ corresponding to $w^k\otimes e_0$
via the isomorphism
\[
 L_k \xrightarrow[\sim]{\psi_1\circ\cdots\circ\psi_k}
 (w^k)/(w^{k+mr-r+1})\otimes L_0.
\]
Since $\pi_k|_{V_k}$ is surjective, we can take an element $e_k\in V_k$ satisfying
$\pi_k(e_k)=\bar{e}_k$.
Then we have
\begin{align*}
 w^l\pi_k(e_k)
 &=
 (\phi_{k+1}\circ\cdots\circ\phi_{k+l})(\pi_{k+l}(e_{k+l}))
 =\pi_k(e_{k+l})
 \quad\quad
 (k+l<r)
 \\
 w^{r-k+l}\pi_k(e_k)
 &=
 (\phi_{k-1}\circ\cdots\circ\phi_{r-1}\circ\phi_r
 \circ(\mathrm{id}\otimes\phi_1)\circ\cdots(\mathrm{id}\otimes\phi_l))
 (z\otimes\pi_l(e_l))
 \\
 &=
 \pi_k(ze_l)
 \quad\quad
 (l<k).
\end{align*}
So the representation matrix of
$\nabla\otimes\mathbb{C}[z]/(z^m)$ with respect to the basis
$e_0,\ldots,e_{r-1}$ is
\[
 \begin{pmatrix}
  \nu_0(z) & z\nu_{r-1}(z) & \cdots & z\nu_1(z) \\
  \tilde{\nu}_1(z) & \nu_0(z)+\frac{dz}{rz} & \cdots & \nu_2(z) \\
  \vdots & \vdots & \ddots & \ \vdots \\
  \tilde{\nu}_{r-1}(z) & \tilde{\nu}_{r-2}(z) & \cdots & \nu_0(z)+\frac{(r-1)dz}{rz}
 \end{pmatrix}.
\]
whose  lower triangular entries
$\tilde{\nu}_1(z),\ldots,\tilde{\nu}_{r-1}(z)$
satisfy $\tilde{\nu}_k(z)\equiv\nu_k(z) \pmod {z^{m-1}dz/z^m}$
and have ambiguities in residue parts
for $1\leq k\leq r-1$.
In any case, we have
$\Tr(\nabla\otimes\mathbb{C}[z]/(z^m))=r\nu_0(z)+(r-1)dz/2z$.
So we get $\nu_0(z)\equiv \nu'_0(z) \pmod{z^mdz/z^m}$.
Thus we have
$\nu(w)=\nu'(w)$ because
\[
 \nu(w)\equiv \nu'(w)
 \pmod {w^{mr-r+1}dw/w^{mr-r+1}}
\]
and both of 
$\nu(w)$ and $\nu'(w)$
belong to $\sum_{l=0}^{mr-r}\mathbb{C}w^ldw/w^{mr-r+1}$.
Thus we get a formal isomorphism
$\pi'\colon (W,\nabla) \stackrel{\sim}\longrightarrow
(\mathbb{C}[[w]],\nabla_{\nu(w)})$.
\end{proof}

\section{Definition and construction of the moduli space
of parabolic connections of generic ramified type}
\label{section:definition-and-construction}

Let $C$ be a smooth projective irreducible curve over
$\Spec\mathbb{C}$ of genus $g$
and $D=\sum_{i=1}^nm_it_i$ be an effective divisor on $C$
with $m_i\geq 1$ for any $i$.
Take integers $r,a,(r^{(i)}_j)^{1\leq i\leq n}_{0\leq j\leq s_i-1}$
satisfying $r>0$, $r^{(i)}_j>0$
and $r=\sum_{j=0}^{s_i-1}r^{(i)}_j$ for any $i$.
If $m_i=1$, we further assume that $r^{(i)}_j=1$ for all $j$.
We take a generator $z_i$ of the maximal ideal of
${\mathcal O}_{C,t_i}$ and take a variable
$w^{(i)}_j$ satisfying $(w^{(i)}_j)^{r^{(i)}_j}=z_i$.
Then the completion $\hat{\mathcal O}_{C,t_i}$ 
of ${\mathcal O}_{C,t_i}$
is isomorphic to the formal power series ring
$\mathbb{C}[[z_i]]$ in $z_i$.
The formal power series ring
$\mathbb{C}[[w^{(i)}_j]]$ in $w^{(i)}_j$
becomes a free $\mathbb{C}[[z_i]]$-module of rank $r^{(i)}_j$.
Note that we have
\[
  dz_i=r^{(i)}_j(w^{(i)}_j)^{r^{(i)}_j-1}dw^{(i)}_j.
\]
We denote by ${\mathcal O}_{m_it_i}$ the structure sheaf of
the effective divisor $m_it_i$ which is considered as a closed subscheme of $C$.
So we have ${\mathcal O}_{m_it_i}\cong\mathbb{C}[z_i]/(z_i^{m_i})$.
We set
\[
 N((r^{(i)}_j),a):=
 \left\{
 \bnu=\Big(\nu^{(i)}_j\big(w^{(i)}_j\big)\Big)^{1\leq i\leq n}_{0\leq j\leq s_i-1} \left|
 \begin{array}{l}
 \displaystyle
 \nu^{(i)}_j(w^{(i)}_j)=\sum_{k=0}^{r^{(i)}_j-1}\nu^{(i)}_{j,k}(z_i)(w^{(i)}_j)^k,
 \\
 \text{$\nu^{(i)}_{j,k}(z_i)\in\sum_{l=0}^{m_i-2}\mathbb{C}z_i^ldz_i/z_i^{m_i}$
 for $1\leq k\leq r^{(i)}_j-1$,}
 \\
 \text{$\nu^{(i)}_{j,0}(z_i)\in\sum_{l=0}^{m_i-1}\mathbb{C}z_i^ldz_i/z_i^{m_i}$ and}
 \\
 \displaystyle
 a+\sum_{i=1}^n\sum_{j=0}^{s_i-1}
 \Big(r^{(i)}_j\res_{z^{(i)}_j}(\nu^{(i)}_{j,0}(z_i))+\frac{r^{(i)}_j-1}{2}\Big)=0
 \end{array}
 \right.
 \right\}
\]
and call an element of $N((r^{(i)}_j),a)$ a ramified exponent.
Note that a ramified exponent
$\bnu\in N((r^{(i)}_j),a)$ is determined by the coefficients of
$\nu^{(i)}_{j,k}(z_i)$.

\begin{definition} \label {def:connection-generic-ramified} \rm
Take  a ramified exponent
$\bnu=\big(\nu^{(i)}_j(w^{(i)}_j)\big)^{1\leq i\leq n}_{0\leq j\leq s_i-1}\in N((r^{(i)}_j),a)$.
We say that a tuple
$(E,\nabla,\{l^{(i)}_j,V^{(i)}_{j,k},L^{(i)}_{j,k},\pi^{(i)}_{j,k},\phi^{(i)}_{j,k}\})$
is a parabolic connection of generic ramified type with the exponent $\bnu$ if
\begin{itemize}
\item[(i)]
$E$ is an algebraic vector bundle on $C$ of rank $r$ and degree $a$,
\item[(ii)] 
$\nabla\colon E\longrightarrow E\otimes\Omega^1_C(D)$ 
is an algebraic connection admitting poles along $D$,
\item[(iii)] 
$E|_{m_it_i}=l^{(i)}_0\supset l^{(i)}_1\supset\cdots\supset l^{(i)}_{s_i-1}\supset l^{(i)}_{s_i}=0$
is a filtration by ${\mathcal O}_{m_it_i}$-submodules
such that 
$l^{(i)}_j/l^{(i)}_{j+1}\cong{\mathcal O}_{m_it_i}^{\oplus r^{(i)}_j}$
and $\nabla|_{m_it_i}(l^{(i)}_j)\subset l^{(i)}_j\otimes\Omega^1_C(D)$
for $1\leq i \leq n$ and $0\leq j \leq s_i-1$,
\item[(iv)] 
$\Big( \big( V^{(i)}_{j,k},L^{(i)}_{j,k},\pi^{(i)}_{j,k}\big)_{0\leq k\leq r^{(i)}_j-1},
\big(\phi^{(i)}_{j,k}\big)_{1\leq k\leq r^{(i)}_j}\Big)$
is a $\nu^{(i)}_j(w^{(i)}_j)$-ramified structure on
$(l^{(i)}_j/l^{(i)}_{j+1},\nabla^{(i)}_j)$
in the sense of Definition \ref {definition: ramified structure},
where 
$\nabla^{(i)}_j\colon l^{(i)}_j/l^{(i)}_{j+1}
\longrightarrow l^{(i)}_j/l^{(i)}_{j+1}\otimes\Omega^1_C(D)$
is the homomorphism induced by $\nabla|_{m_it_i}$. 
\end{itemize}
\end{definition}

\begin{remark} \rm
If the leading terms
$\big\{ \nu^{(i)}_{j,0}(0) \big\}_{0\leq j\leq s_i-1}$ are distinct for any $i$
and if $\nu^{(i)}_{j,1}(0)\neq 0$ for any $i$,
then we can see by Proposition \ref  {prop:generic-2}
that $(E,\nabla)\otimes\widehat{\mathcal O}_{C,t_i} 
\cong\bigoplus_{j=0}^{s_i-1} \big( \mathbb{C}[[w^{(i)}_j]],\nabla_{\nu(w^{(i)}_j)}\big)$
at each point $t_i$.
Without the condition, we can no longer expect the formal structure
to be given by the exponent $\bnu$.
The above formulation is motivated by the author's hope
to construct the moduli space as a canonical degeneration of
the full dimensional moduli space of regular singular connections,
but we should be careful not to confuse the meaning of the moduli.
\end{remark}

\begin{definition}\rm
Take a tuple of rational numbers $\balpha=(\alpha^{(i)}_j)^{1\leq i\leq n}_{1\leq j\leq s_i}$
satisfying
\[
 0<\alpha^{(i)}_1<\alpha^{(i)}_2<\cdots<\alpha^{(i)}_{s_i}<1
\]
for any $i$
and $\alpha^{(i)}_j\neq\alpha^{(i')}_{j'}$ for $(i,j)\neq(i',j')$.
We call $\balpha$ a parabolic weight.
We say that a parabolic connection
$(E,\nabla,\{l^{(i)}_j,V^{(i)}_{j,k},L^{(i)}_{j,k},\pi^{(i)}_{j,k},\phi^{(i)}_{j,k}\})$ of generic ramified type is
$\balpha$-stable (resp.\ $\balpha$-semistable) if the inequality
\begin{align*}
 &\frac{\deg F+\sum_{i=1}^n\sum_{j=1}^{s_i}
 \alpha^{(i)}_j\length((F|_{m_it_i}\cap l^{(i)}_{j-1})/(F|_{m_it_i}\cap l^{(i)}_j))}{\rank F} \\
 &\genfrac{}{}{0pt}{}{<}{(\text{\rm resp. }\leq)}
 \frac{\deg E+\sum_{i=1}^n\sum_{j=1}^{s_i} \alpha^{(i)}_j \length(l^{(i)}_{j-1}/l^{(i)}_j)}{\rank E}
\end{align*}
holds for any non-trivial subbundle $0\neq F\subsetneq E$ satisfying
$\nabla(F)\subset F\otimes\Omega^1_C(D)$.
\end{definition}

\begin{remark}\rm
If  $(E,\nabla,\{l^{(i)}_j,V^{(i)}_{j,k},L^{(i)}_{j,k},\pi^{(i)}_{j,k},\phi^{(i)}_{j,k}\})$
is $\balpha$-stable, we can see that there are only scalar endomorphisms
$u\colon E\longrightarrow E$ satisfying $\nabla\circ u=u\circ\nabla$
and $\nabla|_{m_it_i}(l^{(i)}_j)\subset l^{(i)}_j$ for any $i,j$.
\end{remark}

\begin{remark}\rm
If $s_i=1$ for some $i$ and if the
$\dfrac{w^{(i)}_0dw^{(i)}_0}{(w^{(i)}_0)^{m_ir^{(i)}_0-r^{(i)}_0+1}}$-coefficient of
$\nu^{(i)}_0(w^{(i)}_0)$ is not zero,
then $(E,\nabla)$ is formally irreducible at $t_i$ by
Proposition \ref {proposition: characterization of ramified connection}
and Proposition \ref{prop:generic-2}.
So a parabolic connection
$(E,\nabla,\{l^{(i)}_j,V^{(i)}_{j,k},L^{(i)}_{j,k},\pi^{(i)}_{j,k},\phi^{(i)}_{j,k}\})$
of generic ramified type with the exponent $\bnu$
is $\balpha$-stable with respect to any weight $\balpha$
in this case.
\end{remark}

In the following, we give the definition of the moduli functor of parabolic connections
of generic ramified type
in order to make the meaning of the moduli clear.

\begin{definition} \label {def:moduli-functor} \rm
We define a contravariant functor
${\mathcal M}^{\balpha}_{C,D}((r^{(i)}_j),a):(\sch/N((r^{(i)}_j),a))^o\longrightarrow(\sets)$
from the category $(\sch/N((r^{(i)}_j),a))$ of noetherian schemes over $N((r^{(i)}_j),a)$
to the category $(\sets)$ of sets by setting
\[
 {\mathcal M}^{\balpha}_{C,D}((r^{(i)}_j),a)(S)=
 \left\{
 (E,\nabla,\{l^{(i)}_j,V^{(i)}_{j,k},L^{(i)}_{j,k},\pi^{(i)}_{j,k},\phi^{(i)}_{j,k}\})
 \right\}/\sim,
\]
for a noetherian scheme $S$ over $N((r^{(i)}_j),a)$, where
\begin{itemize}
\item[(i)] $E$ is a vector bundle on $C\times S$ of rank $r$ 
and $\deg(E|_{C\times s})=a$ for any point $s\in S$,
\item[(ii)] $\nabla \colon E\longrightarrow E\otimes\Omega^1_{C\times S/S}(D\times S)$
is a relative connection,
\item[(iii)] $E|_{m_it_i\times S}=l^{(i)}_0\supset l^{(i)}_1\supset\cdots\supset l^{(i)}_{s_i-1}\supset l^{(i)}_{s_i}=0$
is a filtration by ${\mathcal O}_{m_it_i\times S}$-submodules
such that $\nabla|_{m_it_i\times S}(l^{(i)}_j)\subset l^{(i)}_j\otimes\Omega^1_C(D)$
for any $i,j$ and
each $l^{(i)}_j/l^{(i)}_{j+1}$ is a locally free ${\mathcal O}_{m_it_i\times S}$-module
of rank $r^{(i)}_j$, 
\item[(iv)] 
for the homomorphism
$\nabla^{(i)}_j \colon l^{(i)}_j/l^{(i)}_{j+1} \longrightarrow
l^{(i)}_j/l^{(i)}_{j+1}\otimes\Omega^1_C(D)$
induced by $\nabla|_{m_it_i\times S}$,
a tuple
$\Big( \big( V^{(i)}_{j,k},L^{(i)}_{j,k},\pi^{(i)}_{j,k}\big)_{0\leq k\leq r^{(i)}_j-1},
\big(\phi^{(i)}_{j,k}\big)_{1\leq k\leq r^{(i)}_j}\Big)$
is a $S$-flat family of $\nu^{(i)}_j(w^{(i)}_j)$-ramified structure on
$(l^{(i)}_j/l^{(i)}_{j+1},\nabla^{(i)}_j)$
in the following sense,
where $\bnu=(\nu^{(i)}_j(w^{(i)}_j))$ is a family of ramified exponents
determined by the structure morphism
$S\longrightarrow N((r^{(i)}_j),a)$:
\begin{enumerate}
\item[(a)]
$l^{(i)}_j/l^{(i)}_{j+1}=V^{(i)}_{j,0}\supset\cdots\supset 
V^{(i)}_{j,r^{(i)}_j-1}\supset z_i V^{(i)}_{j,0}$ is a filtration by ${\mathcal O}_{m_it_i\times S}$-submodules
such that $V^{(i)}_{j,r^{(i)}_j-1}/z_i V^{(i)}_{j,0}$ and
$V^{(i)}_{j,k}/V^{(i)}_{j,k+1}$ for $0\leq k \leq r^{(i)}_j-2$ are
locally free ${\mathcal O}_{t_i\times S}$-modules of rank one and that
$\nabla^{(i)}_j(V^{(i)}_{j,k})\subset V^{(i)}_{j,k}\otimes\Omega^1_C(D)$ for any $k$,
\item[(b)] 
for $0\leq k\leq r^{(i)}_j-1$,
$\pi^{(i)}_{j,k}\colon V^{(i)}_{j,k}\otimes
{\mathcal O}_S[w^{(i)}_j]/((w^{(i)}_j)^{m_ir^{(i)}_j-r^{(i)}_j+1})\longrightarrow L^{(i)}_{j,k}$
is a locally free quotient ${\mathcal O}_S[w^{(i)}_j]/((w^{(i)}_j)^{m_ir^{(i)}_j-r^{(i)}_j+1})$-module
of rank one
whose restriction
$\pi^{(i)}_{j,k}|_{V^{(i)}_{j,k}}\colon V^{(i)}_{j,k}\longrightarrow L^{(i)}_{j,k}$
is surjective for any $k$
and the diagram
\[
 \begin{CD}
  V^{(i)}_{j,k} @>\pi^{(i)}_{j,k}\big|_{V^{(i)}_{j,k}} >> L^{(i)}_{j,k} \\
  @V\nabla^{(i)}_j VV    @VV \nu^{(i)}_j(w^{(i)}_j)+\frac{k\, dz_i}{r^{(i)}_j z_i} V \\
  V^{(i)}_{j,k}\otimes\Omega^1_{C\times S/S}(D\times S)
  @>(\pi^{(i)}_{j,k}\otimes\mathrm{id})\big|_{V^{(i)}_{j,k}\otimes\Omega^1_{C\times S/S}(D)}>>
  L^{(i)}_{j,k}\otimes\Omega^1_{C\times S/S}(D\times S) 
 \end{CD}
\]
is commutative,
\item[(c)] 
$\phi^{(i)}_{j,r^{(i)}_j}\colon (z_i)/(z_i^{m_i+1})\otimes L^{(i)}_{j,0}\longrightarrow L^{(i)}_{j,r^{(i)}_j-1}$
and
$\phi^{(i)}_{j,k}\colon L^{(i)}_{j,k}\longrightarrow L^{(i)}_{j,k-1}$ for $k=1,\ldots,r^{(i)}_j-1$ are
${\mathcal O}_S[w^{(i)}_j]$-homomorphisms
such that
$\im(\phi^{(i)}_{j,k})=w^{(i)}_j L^{(i)}_{j,k-1}$ for $k=1,\ldots,r^{(i)}_j$ 
and that the diagrams
\[
 \begin{CD}
  V^{(i)}_{j,k}  @>\pi^{(i)}_{j,k}\big|_{V^{(i)}_{j,k}} >>  L^{(i)}_{j,k} \\
  @VVV   @V\phi^{(i)}_{j,k}VV \\
  V^{(i)}_{j,k-1} @>\pi^{(i)}_{j,k-1}\big|_{V^{(i)}_{j,k-1}} >>  L^{(i)}_{j,k-1}
 \end{CD}
\]
for $k=1,\ldots,r^{(i)}_j-1$ and the diagram
\[
 \begin{CD}
  (z_i)/(z_i^{m_i+1})\otimes V^{(i)}_{j,0}
 @>\mathrm{id}\otimes\pi^{(i)}_{j,0}\big|_{V^{(i)}_{j,0}} >>
  (z_i)/(z_i^{m_i+1})\otimes L^{(i)}_{j,0} \\
  @VVV  @VV\phi^{(i)}_{j,r^{(i)}_j}V \\
  V^{(i)}_{j,r^{(i)}_j-1}
  @>\pi^{(i)}_{j,r^{(i)}_j-1}\big|_{V^{(i)}_{j,r^{(i)}_j-1}} >> L^{(i)}_{j,r^{(i)}_j-1}
 \end{CD}
\]
are commutative,
\item[(d)]
there are ${\mathcal O}_S[w^{(i)}_j]$-isomorphisms
$\psi^{(i)}_{j,k}\colon
L^{(i)}_{j,k} \xrightarrow{\sim}
(w^{(i)}_j)/((w^{(i)}_j)^{m_ir^{(i)}_j-r^{(i)}_j+2})\otimes L^{(i)}_{j,k-1}$
for $1\leq k\leq r^{(i)}_j-1$
such that the composition
\[
 L^{(i)}_{j,k} \xrightarrow[\sim]{\psi^{(i)}_{j,k}}
 (w^{(i)}_j)/((w^{(i)}_j)^{m_ir^{(i)}_j-r^{(i)}_j+2})\otimes L^{(i)}_{j,k-1}
 \longrightarrow w^{(i)}_jL^{(i)}_{j,k-1} \hookrightarrow L^{(i)}_{j,k-1}
\]
coincides with $\phi^{(i)}_{j,k}\colon L^{(i)}_{j,k}\longrightarrow L^{(i)}_{j,k-1}$
and that the composition
\[
 \psi^{(i)}_{j,1}\circ\cdots\circ\psi^{(i)}_{j,r^{(i)}_j-1}\circ\phi^{(i)}_{j,r^{(i)}_j}
 \colon
 (z_i)/(z_i^{m_i})\otimes L^{(i)}_{j,0}\longrightarrow
 ((w^{(i)}_j)^{r^{(i)}_j-1})/((w^{(i)}_j)^{m_ir^{(i)}_j})\otimes L^{(i)}_{j,0}
\]
coincides with the homomorphism induced by
$(z_i)/(z_i^{m_i})\longrightarrow ((w^{(i)}_j)^{r^{(i)}_j-1})/((w^{(i)}_j)^{m_ir^{(i)}_j})$,
\end{enumerate}
\item[(v)] for any geometric point $\Spec K \longrightarrow S$, the fiber
$(E,\nabla,\{l^{(i)}_j,V^{(i)}_{j,k},L^{(i)}_{j,k},\pi^{(i)}_{j,k},\phi^{(i)}_{j,k}\})\otimes K$ is
$\balpha$-stable.
\end{itemize}
Here $(E,\nabla,\{l^{(i)}_j,V^{(i)}_{j,k},L^{(i)}_{j,k},\pi^{(i)}_{j,k},\phi^{(i)}_{j,k}\})\sim
(E',\nabla',\{l'^{(i)}_j,V'^{(i)}_{j,k},L'^{(i)}_{j,k},\pi'^{(i)}_{j,k},\phi'^{(i)}_{j,k}\})$
if there are a line bundle ${\mathcal L}$ on $S$ and
\begin{itemize}
\item[($\alpha$)]
isomorphisms 
$\theta\colon E\xrightarrow{\sim} E'\otimes {\mathcal L}$
of vector bundles on $C\times S$,
\item[($\beta$)]
isomorphisms
$\vartheta^{(i)}_{j,k}\colon L^{(i)}_{j,k}\xrightarrow{\sim} L'^{(i)}_{j,k}\otimes {\mathcal L}$
of ${\mathcal O}_S[w^{(i)}_j]/((w^{(i)}_j)^{m_ir^{(i)}_j-r^{(i)}_j+1})$-modules for any $i,j,k$
\end{itemize}
such that the equalities
\begin{align*}
 (\theta\otimes\mathrm{id})\circ\nabla
 &=
 \nabla'\circ\theta
 \\
 \theta|_{m_it_i\times S}(l^{(i)}_j)
 &=
 l'^{(i)}_j\otimes {\mathcal L}
 \quad\quad (\text{for any $i,j$})
\end{align*}
hold and that for the induced isomorphism
$\theta^{(i)}_j\colon l^{(i)}_j/l^{(i)}_{j+1}\xrightarrow{\sim}
(l'^{(i)}_j/l'^{(i)}_{j+1})\otimes {\mathcal L}$,
the equalities
\begin{align*}
 \theta^{(i)}_j(V^{(i)}_{j,k})
 &=
 V'^{(i)}_{j,k}\otimes {\mathcal L}
 \quad\quad (0\leq k\leq r^{(i)}_j-1)
 \\
 (\pi'^{(i)}_{j,k}\otimes 1)|_{V'^{(i)}_{j,k}\otimes {\mathcal L}}\circ\theta^{(i)}_j|_{V^{(i)}_{j,k}}
 &=
 \vartheta^{(i)}_{j,k}\circ \pi^{(i)}_{j,k}|_{V^{(i)}_{j,k}}
 \quad\quad
 (0\leq k\leq r^{(i)}_j-1)
 \\
 \vartheta^{(i)}_{j,k-1}\circ\phi^{(i)}_{j,k}
 &=
 (\phi'^{(i)}_{j,k}\otimes 1)\circ\vartheta^{(i)}_{j,k}
 \quad\quad (1\leq k\leq r^{(i)}_j-1)
 \\
 \vartheta^{(i)}_{j,r^{(i)}_j-1}\circ(\mathrm{id}_{(z)}\otimes\phi^{(i)}_{j,0})
 &=
 \phi'^{(i)}_{j,r^{(i)}_j}\circ(\mathrm{id}_{(z)}\otimes\vartheta^{(i)}_{j,0})
\end{align*}
hold.
\end{definition}

\begin{theorem} \label {thm-existence-moduli-generic}
 There exists a relative coarse moduli scheme $M^{\balpha}_{C,D}((r^{(i)}_j),a)$
 of $\balpha$-stable parabolic connections
 of generic ramified type over $N((r^{(i)}_j),a)$.
 Moreover,
 $M^{\balpha}_{C,D}((r^{(i)}_j),a) \longrightarrow N((r^{(i)}_j),a)$
 is a quasi-projective morphism.
\end{theorem}

\begin{proof}
We consider the moduli functor
${\mathcal X}\colon (\sch/\mathbb{C})^o\longrightarrow(\sets)$
defined by
\[
 {\mathcal X}(S)=
 \left\{
 (E,\nabla,\{l^{(i)}_j\})
 \right\}/\sim
\]
for a noetherian scheme $S$ over $\Spec\mathbb{C}$, where
\begin{itemize}
\item[(a)] $E$ is a vector bundle on $C\times S$ of rank $r$ and
$\deg(E|_{C\times\{s\}})=a$ for any $s\in S$,
\item[(b)] $\nabla\colon E\longrightarrow E\otimes\Omega^1_{C\times S/S}(D\times S)$
is a relative connection,
\item[(c)] $E|_{m_it_i\times S}=l^{(i)}_0\supset l^{(i)}_1\supset\cdots\supset l^{(i)}_{s_i-1}\supset l^{(i)}_{s_i}=0$
is a filtration by ${\mathcal O}_{m_it_i\times S}$-submodules
such that
$\nabla|_{m_it_i\times S}(l^{(i)}_j)\subset l^{(i)}_j\otimes\Omega^1_{C\times S/S}(D\times S)$
for any $i,j$
and that each $l^{(i)}_j/l^{(i)}_{j+1}$
is a locally free ${\mathcal O}_{m_it_i\times S}$-module of rank $r^{(i)}_j$,
\item[(d)] for any geometric point $s$ of $S$ and for any non-trivial subbundle
$0\neq F\subsetneq E|_{C\times s}$
satisfying $\nabla(F)\subset F\otimes \Omega^1_{C\times s}(D\times s)$, the inequality
\begin{align*}
 &\frac{\deg F+\sum_{i=1}^n\sum_{j=1}^{s_i} \alpha^{(i)}_j
 \length((F|_{m_it_i\times s}\cap l^{(i)}_{j-1}|_{m_it_i\times s})/(F|_{m_it_i\times s}\cap l^{(i)}_j|_{m_it_i\times s}))}
 {\rank F} \\
 &<
 \frac{\deg (E|_{C\times s})+\sum_{i=1}^n\sum_{j=1}^{s_i} \alpha^{(i)}_j
 \length(l^{(i)}_{j-1}|_{m_it_i\times s}/l^{(i)}_j|_{m_it_i\times s})}{\rank E}
\end{align*}
holds.
\end{itemize}
Here $(E,\nabla,\{l^{(i)}_j\})\sim(E',\nabla',\{l'^{(i)}_j\})$
if there are a line bundle ${\mathcal L}$ on $S$ and
an isomorphism 
$\theta\colon E\xrightarrow{\sim}E'\otimes {\mathcal L}$
satisfying $\nabla'\theta=(\theta\otimes\mathrm{id})\nabla$ and
$\theta|_{m_it_i\times S}(l^{(i)}_j)=l'^{(i)}_j\otimes {\mathcal L}$ for any $i,j$.
By a similar argument to that of \cite[Theorem 2.1]{Inaba-Saito},
we can realize a coarse moduli scheme $X$ of ${\mathcal X}$
as a locally closed subscheme of the moduli space of stable parabolic $\Lambda^1_D$-triples
constructed in \cite[Theorem 5.1]{IIS-1}.
So $X$ is quasi-projective and there is a universal family
$(\tilde{E},\tilde{\nabla},\{\tilde{l}^{(i)}_j\})$
over some \'etale covering $X'\longrightarrow X$.
We consider the moduli functor
${\mathcal Y}\colon(\sch/X')^o\longrightarrow(\sets)$
defined by
\[
 {\mathcal Y}(S)=\left\{
 (V^{(i)}_{j,k},L^{(i)}_{j,k},\pi^{(i)}_{j,k})^{1\leq i\leq n}_{0\leq j\leq s_i-1,0\leq k\leq r^{(i)}_j-1}
 \right\}
\]
for a noetherian scheme $S$ over $X'$, where
\begin{itemize}
\item[(e)] $l^{(i)}_j/l^{(i)}_{j+1}\otimes{\mathcal O}_S=V^{(i)}_{j,0}\supset V^{(i)}_{j,1}
\supset\cdots\supset V^{(i)}_{j,r^{(i)}_j-1}\supset z_i V^{(i)}_{j,0}$
is a filtration by ${\mathcal O}_{m_it_i\times S}$-submodules
such that each $V^{(i)}_{j,0}/V^{(i)}_{j,k}$ is flat over $S$
and
$\dim_{k(s)}((V^{(i)}_{j,k}/V^{(i)}_{j,k+1})\otimes k(s))
=\dim_{k(s)}((V^{(i)}_{j,r^{(i)}_j-1}/z_iV^{(i)}_{j,0})\otimes k(s))=1$
for any $s\in S$,
\item[(f)] $\pi^{(i)}_{j,k} \colon 
V^{(i)}_{j,k}\otimes{\mathcal O}_S[w^{(i)}_j]/((w^{(i)}_j)^{m_ir^{(i)}_j-r^{(i)}_j+1})
\longrightarrow L^{(i)}_{j,k}$
is a quotient ${\mathcal O}_S[w^{(i)}_j]$-module such that
$L^{(i)}_{j,k}$ is a locally free ${\mathcal O}_S[w^{(i)}_j]/((w^{(i)}_j)^{m_ir^{(i)}_j-r^{(i)}_j+1})$-module
of rank one and the restriction
$\pi^{(i)}_{j,k}|_{V^{(i)}_{j,k}} \colon V^{(i)}_{j,k}\longrightarrow L^{(i)}_{j,k}$
is surjective for $k=0,\ldots,r^{(i)}_j-1$.
\end{itemize}
We can see that ${\mathcal Y}$ can be represented by a locally closed subscheme
$Y$ of a product of flag schemes over $X'$.
Let $(\tilde{V}^{(i)}_{j,k},\tilde{L}^{(i)}_{j,k},\tilde{\pi}^{(i)}_{j,k})$ be a universal family over $Y$.
We also take a pull-back $\tilde{\bnu}=(\tilde{\nu}^{(i)}_j)$ to $Y\times N((r^{(i)}_j),a)$
of the universal family over $N((r^{(i)}_j),a)$.
There is a maximal locally closed subscheme $Y'$ of $Y\times N((r^{(i)}_j),a)$ such that
the composition
\[
 \tilde{V}^{(i)}_{j,k}\otimes{\mathcal O}_{Y'}[w^{(i)}_j]/((w^{(i)}_j)^{m_ir^{(i)}_j-r^{(i)}_j+1})
 \longrightarrow
 \tilde{V}^{(i)}_{j,k-1}\otimes{\mathcal O}_{Y'}[w^{(i)}_j]/((w^{(i)}_j)^{m_ir^{(i)}_j-r^{(i)}_j+1})
 \xrightarrow{\tilde{\pi}^{(i)}_{j,k-1}}
 (\tilde{L}^{(i)}_{j,k-1})_{Y'}
\]
factors through an ${\mathcal O}_{Y'}[w^{(i)}_j]/((w^{(i)}_j)^{m_ir^{(i)}_j-r^{(i)}_j+1})$-homomorphism
\[
 \tilde{\phi}^{(i)}_{j,k}\colon (\tilde{L}^{(i)}_{j,k})_{Y'} \longrightarrow (\tilde{L}^{(i)}_{j,k-1})_{Y'}
\]
whose image is $w^{(i)}_j(\tilde{L}^{(i)}_{j,k-1})_{Y'}$
for $1\leq k\leq r^{(i)}_j-1$,
the composition
\[
  (z_i)/(z_i^{m_i+1})\otimes \tilde{V}^{(i)}_{j,0}
  \longrightarrow
  \tilde{V}^{(i)}_{j,r^{(i)}_j-1}\otimes
  {\mathcal O}_{Y'}[w^{(i)}_j]/((w^{(i)}_j)^{m_ir^{(i)}_j-r^{(i)}_j+1})
  \xrightarrow{\tilde{\pi}^{(i)}_{j,r^{(i)}_j-1}}
 ( \tilde{L}^{(i)}_{j,r^{(i)}_j-1} )_{Y'}
\].
factors through an ${\mathcal O}_{Y'}[w^{(i)}_j]/((w^{(i)}_j)^{m_ir^{(i)}_j-r^{(i)}_j+1})$-homomorphism
\[
 \tilde{\phi}^{(i)}_{r^{(i)}_j}
 \colon (z)/(z^{m_i+1})\otimes(\tilde{L}^{(i)}_{j,0})_{Y'}
 \longrightarrow
 (\tilde{L}^{(i)}_{j,r^{(i)}_j-1})_{Y'}
\]
whose image is $w^{(i)}_j\tilde{L}^{(i)}_{j,r^{(i)}_j-1}$
and the diagrams
\[
 \begin{CD}
  \tilde{V}^{(i)}_{j,k}\otimes{\mathcal O}_{Y'}[w^{(i)}_j]/((w^{(i)}_j)^{m_ir^{(i)}_j-r^{(i)}_j+1})
  @>\tilde{\pi}^{(i)}_{j,k}\otimes\mathrm{id}>> \tilde{L}^{(i)}_{j,k}\otimes{\mathcal O}_{Y'} \\
  @V\nabla^{(i)}_j\otimes\mathrm{id}VV   
  @VV\tilde{\nu}^{(i)}_j(w^{(i)}_j)+\frac{k\, dz_i}{r^{(i)}_j z_i} V \\
  \tilde{V}^{(i)}_{j,k}\otimes\Omega^1_C(D)\otimes
  {\mathcal O}_{Y'}[w^{(i)}_j]/((w^{(i)}_j)^{m_ir^{(i)}_j-r^{(i)}_j+1})
  @>\tilde{\pi}^{(i)}_{j,k}\otimes\mathrm{id}>>
  \tilde{L}^{(i)}_{j,k}\otimes\Omega^1_C(D)\otimes{\mathcal O}_{Y'} \\
 \end{CD}
\]
are commutative for all $i,j,k$.

Consider the product
\[
 Z=\prod_{i,j,k}\Spec \left( \mathrm{Sym}
 \bigg( {\mathcal Hom}_{Y'}\Big(
 {\mathcal Hom}_{{\mathcal O}_{Y'}[w^{(i)}_j]}
 \big( \tilde{L}^{(i)}_{j,k}  , 
 \big(w^{(i)}_j \big)\big/ \big((w^{(i)}_j)^{m_ir^{(i)}_j-r^{(i)}_j+2}\big)
 \otimes \tilde{L}^{(i)}_{j,k-1} \big) , 
 {\mathcal O}_{Y'} \Big)\bigg) \right)
\]
of affine space bundles.
There are universal sections
\[
 \tilde{\psi}^{(i)}_{j,k}\colon
 (\tilde{L}^{(i)}_{j,k})_Z\longrightarrow
  \left( (w^{(i)}_j)/((w^{(i)}_j)^{m_ir^{(i)}_j-r^{(i)}_j+2})\otimes \tilde{L}^{(i)}_{j,k-1} \right)_Z
\]
for $1\leq k\leq r^{(i)}_j-1$.
Let $Z'$ be the maximal locally closed subscheme of $Z$
such that $(\tilde{\psi}^{(i)}_{j,k})_{Z'}$ are isomorphic and that
the composition
\[
 (\tilde{L}^{(i)}_{j,k})_{Z'} \xrightarrow {\tilde{\psi}^{(i)}_{j,k}}
  \left( (w^{(i)}_j)/((w^{(i)}_j)^{m_ir^{(i)}_j-r^{(i)}_j+2})\otimes \tilde{L}^{(i)}_{j,k-1} \right)_{Z'}
  \longrightarrow \left(w^{(i)}_j \tilde{L}^{(i)}_{j,k-1}\right)_{Z'}
\]
coincides with $(\tilde{\phi}^{(i)}_{j,k})_{Z'}$ for $1\leq k\leq r^{(i)}_j-1$ and any $i,j$.
Consider the group scheme $G$ over $Y'$ whose set of $S$-valued points is
\[
 G(S):=
 \left\{ (1+a^{(i)}_{j,k})\in \prod_{i,j,k}{\mathcal O}_S[w^{(i)}_j]/((w^{(i)}_j)^{m_ir^{(i)}_j-r^{(i)}_j+1})
 \: \middle| \; 
 \begin{array}{l}
 \text{$a^{(i)}_{j,k}\in ((w^{(i)}_j)^{m_ir^{(i)}_j-r^{(i)}_j-1})$ and}
 \\
 \text{$a^{(i)}_{j,k-1}- a^{(i)}_{j,k} \in ((w^{(i)}_j)^{m_ir^{(i)}_j-r^{(i)}_j})$}
 \\
 \text{for $1\leq k\leq r^{(i)}_j-1$}
 \end{array}
 \right\}
\]
for any $Y'$-scheme $S$.
Then we can see that $Z'\longrightarrow Y'$ is a principal $G$-bundle.
Let
$\Sigma$ be the maximal closed subscheme of
$Z'$ such that
the composition
\[
 (\tilde{\psi}^{(i)}_{j,1}\circ\cdots\circ \tilde{\psi}^{(i)}_{j,r^{(i)}_j-1}
 \circ \tilde{\phi}^{(i)}_{j,r^{(i)}_j})_{\Sigma}
 \colon
 (z_i)/(z_i^{m_i})\otimes (\tilde{L}^{(i)}_{j,0})_{\Sigma}\longrightarrow
 ((w^{(i)}_j)^{r^{(i)}_j-1})/((w^{(i)}_j)^{m_ir^{(i)}_j})\otimes (\tilde{L}^{(i)}_{j,0})_{\Sigma}
\]
coincides with the homomorphism induced by
$(z_i)/(z_i^{m_i})\longrightarrow ((w^{(i)}_j)^{r^{(i)}_j-1})/((w^{(i)}_j)^{m_ir^{(i)}_j})$.
Then $\Sigma$ is preserved by the action of $G$.
So $\Sigma$ descends to a closed subscheme
$M'$ of $Y'$.

By the construction, we can see that $M'\longrightarrow X'\times N((r^{(i)}_j),a)$ descends to
a quasi-projective morphism
$M\longrightarrow X\times N((r^{(i)}_j),a)$, that is,
$M'\cong M\times_{X\times N((r^{(i)}_j),a)} (X'\times N((r^{(i)}_j),a))$.
Then $M$ is nothing but the desired moduli space
$M^{\balpha}_{C,D}((r^{(i)}_j),a)$.
\end{proof}

\section{Smoothness and dimension of the moduli space of parabolic connections
of generic ramified type}\label{section:smoothness-and-dimension}

The aim of this section is to prove the following theorem:

\begin{theorem} \label {thm:generic-smooth}
The structure morphism
$M^{\balpha}_{C,D}((r^{(i)}_j),a)\longrightarrow N((r^{(i)}_j),a)$
is a smooth morphism and its fiber $M^{\balpha}_{C,D}((r^{(i)}_j),a)_{\bnu}$
over $\bnu\in N((r^{(i)}_j),a)$ is of equi-dimension
\[
 2r^2(g-1)+2+\sum_{i=1}^n(r^2-r)m_i
\]
if $M^{\balpha}_{C,D}((r^{(i)}_j),a)_{\bnu}\neq\emptyset$.
\end{theorem}

Before proving the theorem, we prepare a complex which determines a deformation theory
of the moduli space of irregular connections of generic ramified type.
From the construction of the moduli space $M^{\balpha}_{C,D}((r^{(i)}_j),a)$
in Theorem \ref {thm-existence-moduli-generic},
we can see that there exists an \'etale surjective morphism
$M'\longrightarrow M^{\balpha}_{C,D}((r^{(i)}_j),a)$
and a universal family
$(E_{M'},\nabla_{M'},\{ l^{(i)}_{M,j},\tilde{V}^{(i)}_{M',j,k},\tilde{L}^{(i)}_{M',j,k},
\tilde{\pi}^{(i)}_{M',j,k},\phi^{(i)}_{M',j,k}\})$
 of irregular connections of generic ramified type.
Define a complex
${\mathcal F}^{\bullet}_{M'}$ of sheaves on $C\times M'$ by
\begin{align} \label {equation: definition of tangent complex}
\begin{split}
 {\mathcal F}^0_{M'} &=
 \left\{
  u\in{\mathcal End}(E_{M'}) \left|
  \begin{array}{l}
   \text{$u|_{m_it_i\times M'}(l^{(i)}_{M',j})\subset l^{(i)}_{M',j}$ for any $i,j$} \\
   \text{and for the induced homomorphism} \\
   \text{$u^{(i)}_j:l^{(i)}_{M',j}/l^{(i)}_{M',j+1} 
   \longrightarrow l^{(i)}_{M',j}/l^{(i)}_{M',j+1}$,} \\
   \text{$u^{(i)}_j(V^{(i)}_{M',j,k})\subset V^{(i)}_{M',j,k}$ and} \\
   \text{$\bar{\pi}^{(i)}_{M',j,k}((u^{(i)}_j\otimes\mathrm{id})
   (\ker\bar{\pi}^{(i)}_{M',j,k}))=0$ for any $k$}  \\
  \end{array}
 \right.\right\}, \\
 {\mathcal F}^1_{M'} &=
 \left\{
  v\in{\mathcal End}(E_{M'})\otimes\Omega^1_C(D) \left|
  \begin{array}{l}
   \text{$v|_{m_it_i\times M'}(l^{(i)}_{M',j})\subset l^{(i)}_{M',j}\otimes\Omega^1_C(D)$ for any $i,j$} \\
   \text{and for the induced homomorphism} \\
   \text{$v^{(i)}_j:l^{(i)}_{M',j}/l^{(i)}_{M',j+1}\longrightarrow
   l^{(i)}_{M',j}/l^{(i)}_{M',j+1}\otimes\Omega^1_C(D)$,} \\
   \text{$v^{(i)}_j(V^{(i)}_{M',j,k})\subset V^{(i)}_{M',j,k}\otimes\Omega^1_C(D)$
   for any $k$ and} \\
   \text{$(\bar{\pi}^{(i)}_{j,k}\otimes\mathrm{id})(v^{(i)}_j(V^{(i)}_{M',j,k}))=0$ for any $k$}
  \end{array}
 \right.\right\}, \\
 & \nabla_{{\mathcal F}^{\bullet}_{M'}}\colon {\mathcal F}^0_{M'}
 \ni u \ \mapsto \ \nabla u-u\nabla\in{\mathcal F}^1_{M'}.
\end{split}
\end{align}

For the proof of Theorem \ref{thm:generic-smooth},
we use the morphism
\begin{align} \label {equation: determinant morphism}
\begin{split}
 \det \ \colon \  M^{\balpha}_{C,D}((r^{(i)}_j),a)  \ \longrightarrow \ &
 M_{C,D}(1,a)\times_{N((1),a)} N((r^{(i)}_j),a) \\
 (E,\nabla,\{l^{(i)}_j,V^{(i)}_{j,k},L^{(i)}_{j,k},\pi^{(i)}_{j,k},\phi^{(i)}_{j,k}\})
 \mapsto  & \ (\det(E,\nabla),\bnu),
 \end{split}
\end{align}
where $M_{C,D}(1,a)$ is the moduli space of pairs $(L,\nabla_L)$
of a line bundle $L$ on $C$ of degree $a$ and a connection
$\nabla_L\colon L\longrightarrow L\otimes\Omega^1_C(D)$
and we set
\begin{gather*}
 N((1),a)=\left.\left\{
 (\lambda^{(i)})_{1\leq i\leq n}\in\mathbb{C}[z_i]/(z_i^{m_i})
 \dfrac{dz_i}{z_i^{m_i}}\right|
 a+\sum_{i=1}^n\res(\lambda^{(i)})=0\right\} \\
 M_{C,D}(1,a)\ \ni \ (L,\nabla_L) \ \mapsto \ \left(\nabla|_{m_it_i}\right)_{1\leq i\leq n}
 \ \in \ N((1),a) \\
 N((r^{(i)}_j),a)\ \ni \ \big( \nu^{(i)}_j(w^{(i)}_j) \big) \ \mapsto \
 \bigg(\sum_{j=0}^{s_i-1}
 \Big( r^{(i)}_j\nu^{(i)}_{j,0}(z_i)+(r^{(i)}_j-1)dz_i/2z_i\Big)\bigg)_{1\leq i\leq n}\in \ N((1),a).
\end{gather*}
Recall that $\nu^{(i)}_{j,0}(z_i)$ is determined from
$\nu^{(i)}_j(w^{(i)}_j)$ by the equality
\[
 \nu^{(i)}_j(w^{(i)}_j)=
 \nu^{(i)}_{j,0}(z_i)+\nu^{(i)}_{j,1}(z_i)w^{(i)}_j+\cdots+\nu^{(i)}_{j,r^{(i)}_j-1}(z_i)(w^{(i)}_j)^{r^{(i)}_j-1}.
\]
Note that $M_{C,D}(1,a)$ is an affine space bundle over the Jacobian variety of $C$
and so it is smooth over $N((1),a)$.
So it is sufficient to prove the following proposition 
for the smoothness of $M^{\balpha}_{C,D}((r^{(i)}_j),a)$.

\begin{proposition}\label{prop:det-smooth}
The morphism
$\det \colon
M^{\balpha}_{C,D}((r^{(i)}_j),a)\longrightarrow M_{C,D}(1,a)\times_{N((1),a)} N((r^{(i)}_j),a)$
defined in (\ref {equation: determinant morphism}) is a smooth morphism.
\end{proposition}

\begin{proof}
Take an Artinian local ring $A$ over $\mathbb{C}$ with the maximal ideal $\mathfrak{m}$
and an ideal $I$ of $A$ satisfying $\mathfrak{m}I=0$.
Assume that a commutative diagram
\[
 \begin{CD}
  \Spec A/I @>f>> {\mathcal M}^{\balpha}_{C,D}((r^{(i)}_j),a) \\
  @VVV   @V\det VV \\
  \Spec A @>g>> M_{C,D}((1),a)\times_{N((1),a)}N((r^{(i)}_j),a)
 \end{CD}
\]
is given.
The morphism $g$ corresponds to a tuple
$((L,\nabla_L),\tilde{\bnu})$,
where $L$ is a line bundle on $C\times\Spec A$,
$\nabla_L\colon L\longrightarrow L\otimes
\Omega^1_{C\times\Spec A/\Spec A}(D\times\Spec A)$ is a relative connection
and $\tilde{\bnu}=(\tilde{\nu}^{(i)}_{j,k})\in N((r^{(i)}_j),a)(A)$ satisfies
\[
 \nabla_L|_{m_it_i\times\Spec A}=
 \sum_{j=0}^{s_i-1} \Big(
 r^{(i)}_j\tilde{\nu}^{(i)}_{j,0}(z_i) +(r^{(i)}_j-1)dz_i/2z_i \Big).
\]
If we put $\bnu=(\nu^{(i)}_j(z_i)):=\tilde{\bnu}\otimes A/I$,
$f$ corresponds to a flat family
$(E,\nabla,\{l^{(i)}_j,V^{(i)}_{j,k},L^{(i)}_{j,k},\pi^{(i)}_{j,k},\phi^{(i)}_{j,k}\})$
of parabolic connections of generic ramified type
over $A/I$ with the exponent $\bnu$.
We can choose an isomorphism
$\psi^{(i)}_{j,k}\colon
L^{(i)}_{j,k} \xrightarrow{\sim}
(w^{(i)}_j)\otimes L^{(i)}_{j,k-1}$
of $(A/I)[w^{(i)}_j]/((w^{(i)}_j)^{m_ir^{(i)}_j-r^{(i)}_j+1})$-modules
such that the composition
$L^{(i)}_{j,k} \xrightarrow[\sim]{\psi^{(i)}_{j,k}}
(w^{(i)}_j)\otimes L^{(i)}_{j,k-1}
\longrightarrow w^{(i)}_j L^{(i)}_{j,k-1}\hookrightarrow L^{(i)}_{j,k-1}$
coincides with $\phi^{(i)}_{j,k}$.

Take an affine open covering
$\{U_{\alpha}\}$ of $C$ such that $\sharp\{\alpha|t_i\in U_{\alpha}\}=1$ for any $i$
and $\sharp\{i|t_i\in U_{\alpha}\}\leq 1$ for any $\alpha$.
For each $i$, we take $\alpha$ with $t_i\in U_{\alpha}$.
By shrinking $U_{\alpha}$ if necessary, we can take a basis
$\{ e^{(i)}_{j,k} \}$ of $E|_{U_{\alpha}\times\Spec A/I}$
with the following properties
\begin{itemize}
\item
$l^{(i)}_j$ is generated by 
$\left\{ e^{(i)}_{j',k}|_{m_it_i\times \Spec A/I} \, 
\middle| \, j'\geq j, 0\leq k\leq r^{(i)}_j-1 \right\}$
\item
the induced class 
$f^{(i)}_{j,k}=\overline{e^{(i)}_{j,k}|_{m_it_i\times\Spec A/I}}\in l^{(i)}_j/l^{(i)}_{j+1}$
belongs to $V^{(i)}_{j,k}$
and
\item
the image of 
$\pi^{(i)}_{j,k}(f^{(i)}_{j,k})\in L^{(i)}_{j,k}$ in $((w^{(i)}_j)^k)\otimes L^{(i)}_{j,0}$
via the isomorphism
\[
 L^{(i)}_{j,k}\xrightarrow[\sim]{\psi^{(i)}_{j,k}}
 (w^{(i)}_j)\otimes L^{(i)}_{j,k-1}
 \xrightarrow[\sim]{\psi^{(i)}_{j,k-1}}
 ((w^{(i)}_j)^2)\otimes L^{(i)}_{j,k-2}
 \xrightarrow[\sim]{\psi^{(i)}_{j,k-2}}
 \cdots \xrightarrow[\sim]{\psi^{(i)}_{j,1}}
 ((w^{(i)}_j)^k)\otimes L^{(i)}_{j,0}
\]
is equal to $(w^{(i)}_j)^k\otimes\pi^{(i)}_{j,0}(f^{(i)}_{j,0})$.
\end{itemize}
From the commutativity of the diagram in
Definition \ref  {def:moduli-functor} (b),
we can see that
the representation matrix of
the homomorphism
$\nabla^{(i)}_j\colon
V^{(i)}_{j,0}\longrightarrow
V^{(i)}_{j,0}\otimes\Omega^1_C(D)$
induced by $\nabla|_{m_it_i\times\Spec A/I}$,
with respect to the basis
$f^{(i)}_{j,0},\ldots,f^{(i)}_{j,r^{(i)}_j-1}$,
is given by
\[
 \begin{pmatrix}
  \nu^{(i)}_{j,0}(z_i) & z_i\nu^{(i)}_{j,r^{(i)}_j-1}(z_i) & \cdots & z_i\nu^{(i)}_{j,1}(z_i) \\
  \nu'^{(i)}_{j,1}(z_i) & \nu^{(i)}_{j,0}(z_i)+\frac{1}{r^{(i)}_j}\frac{dz_i}{z_i}
   & \cdots & z_i\nu^{(i)}_{j,2}(z_i) \\
  \vdots & \vdots & \ddots & \vdots \\
  \nu'^{(i)}_{j,r^{(i)}_j-1}(z_i) & \nu'^{(i)}_{j,r^{(i)}_j-2}(z_i) 
  & \cdots & \nu^{(i)}_{j,0}(z_i)+\frac{r^{(i)}_j-1}{r^{(i)}_j}\frac{dz_i}{z_i}
 \end{pmatrix},
\]
where $\nu'^{(i)}_{j,k}(z_i)$ satisfies
$\nu'^{(i)}_{j,k}(z_i)\equiv \nu^{(i)}_{j,k}(z_i) \pmod{z_i^{m_i-1}dz_i/z_i^{m_i}}$
for $k=1,\ldots,r^{(i)}_j-1$.

We take a free $A[w^{(i)}_j]/((w^{(i)}_j)^{m_ir^{(i)}_j-r^{(i)}_j+1})$-module
$\tilde{L}^{(i)}_{j,0}$ of rank one
which is a lift of $L^{(i)}_{j,0}$.
Then we can determine lifts
$\tilde{L}^{(i)}_{j,k}$ of $L^{(i)}_{j,k}$
as free $A[w^{(i)}_j]/((w^{(i)}_j)^{m_ir^{(i)}_j-r^{(i)}_j+1})$-modules of rank one,
together with lifts
$\tilde{\psi}^{(i)}_{j,k}\colon
\tilde{L}^{(i)}_{j,k}\xrightarrow{\sim}
(w^{(i)}_j)\otimes\tilde{L}^{(i)}_{j,k-1}$
of $\psi^{(i)}_{j,k}$
for $k=1,\ldots,r^{(i)}_j-1$.
We define $\tilde{\phi}^{(i)}_{j,k}$ as the composition
$\tilde{\phi}^{(i)}_{j,k}\colon
\tilde{L}^{(i)}_{j,k}\xrightarrow[\sim]{\tilde{\psi}^{(i)}_{j,k}}
(w^{(i)}_j)\otimes\tilde{L}^{(i)}_{j,k-1}
\longrightarrow w^{(i)}_j\tilde{L}^{(i)}_{j,k-1}$
and define 
$\tilde{\phi}^{(i)}_{r^{(i)}_j}\colon  (z_i)\otimes \tilde{L}^{(i)}_{j,0}
\longrightarrow\tilde{L}^{(i)}_{r^{(i)}_j-1}$
as the composition
\[
 (z_i)\otimes \tilde{L}^{(i)}_{j,0}\longrightarrow 
 ((w^{(i)}_j)^{r^{(i)}_j-1})\otimes \tilde{L}^{(i)}_{j,0}
 \xrightarrow[\sim]{(\tilde{\psi}^{(i)}_{j,1})^{-1}}
 ((w^{(i)}_j)^{r^{(i)}_j-2})\otimes \tilde{L}^{(i)}_{j,1}
 \xrightarrow[\sim]{(\tilde{\psi}^{(i)}_{j,r^{(i)}_j-1})^{-1}\circ\cdots\circ(\tilde{\psi}^{(i)}_{j,2})^{-1}}
 \tilde{L}^{(i)}_{r^{(i)}_j-1}.
\]

Choose a lift $\epsilon^{(i)}_{j,0}\in \tilde{L}^{(i)}_{j,0}$
of the generator $\pi^{(i)}_{j,0}(f^{(i)}_{j,0})$ of  $L^{(i)}_{j,0}$.
Let $\epsilon^{(i)}_{j,k}$ be the element of $\tilde{L}^{(i)}_{j,k}$
corresponding to $(w^{(i)}_j)^k\otimes \epsilon^{(i)}_{j,0}$
via the isomorphism
\[
 \tilde{L}^{(i)}_{j,k}
 \xrightarrow[\sim]{\tilde{\psi}^{(i)}_{j,k}}
 (w^{(i)}_j)\otimes\tilde{L}^{(i)}_{j,k-1}
 \xrightarrow[\sim]{\tilde{\psi}^{(i)}_{j,k-1}}
 ((w^{(i)}_j)^2)\otimes \tilde{L}^{(i)}_{j,k-2}
 \xrightarrow[\sim]{\tilde{\psi}^{(i)}_{j,k-2}}
 \cdots\xrightarrow[\sim]{\tilde{\psi}^{(i)}_{j,1}}
 ((w^{(i)}_j)^k)\otimes \tilde{L}^{(i)}_{j,0}.
\]
We take a free ${\mathcal O}_{U_{\alpha}\times \Spec A}$-module
$E_{\alpha}$ of rank $r$ with a basis 
$\{\tilde{e}^{(i)}_{j,k}\}_{0\leq j\leq s_i-1,0\leq k\leq r^{(i)}_j-1}$
and an isomorphism
$E_{\alpha}\otimes A/I\xrightarrow[\sim]{\tau_{\alpha}}
E|_{U_{\alpha}\times\Spec A/I}$
sending $\tilde{e}^{(i)}_{j,k}\otimes A/I$ to $e^{(i)}_{j,k}$.
We define $\tilde{l}^{(i)}_j\subset E_{\alpha}|_{m_it_i\times\Spec A}$
as the submodule generated by
$\left\{ \tilde{e}^{(i)}_{j',k} \, \middle| \, j'\geq j\right\}$.
Let $\tilde{f}^{(i)}_{j,k}$ be the image of
$\tilde{e}^{(i)}_{j,k}$ in $\tilde{l}^{(i)}_j/\tilde{l}^{(i)}_{j+1}$.
We define $\tilde{V}_{j,k}$ as the submodule of
$\tilde{l}^{(i)}_j/\tilde{l}^{(i)}_{j+1}$
generated by
$\tilde{f}^{(i)}_{j,k},\tilde{f}^{(i)}_{j,k+1},\ldots,\tilde{f}^{(i)}_{j,r^{(i)}_j-1},
z_i\tilde{f}^{(i)}_{j,0},\ldots,z_i\tilde{f}^{(i)}_{j,k-1}$.

Define a homomorphism
\[
 \tilde{\pi}^{(i)}_{j,k}
 \colon
 \tilde{V}^{(i)}_{j,k}\otimes A[w^{(i)}_j]/((w^{(i)}_j)^{m_ir^{(i)}_j-r^{(i)}_j+1})
 \longrightarrow \tilde{L}^{(i)}_{j,k}
\]
by setting
\begin{align*}
 \tilde{\pi}^{(i)}_{j,k}(\tilde{f}^{(i)}_{j,l})
 &:=
 \begin{cases}
  \epsilon^{(i)}_{j,k} & (\text{when $l=k$}) \\
  (\tilde{\phi}^{(i)}_{j,k+1}\circ\tilde{\phi}^{(i)}_{j,l})(\epsilon^{(i)}_{j,l})
  =(w^{(i)}_j)^{l-k}\epsilon^{(i)}_{j,k}
  & (\text{when $l>k$})
 \end{cases}
  \\
  \tilde{\pi}^{(i)}_{j,k} (z_i\tilde{f}^{(i)}_{j,l})
  &:=
  \Big(\tilde{\phi}_{j,k+1}\circ\cdots\circ\tilde{\phi}^{(i)}_{j,r^{(i)}_j}
  \circ(\mathrm{id}_{(z_i)}\otimes\tilde{\phi}^{(i)}_{j,1})
  \circ\cdots\circ(\mathrm{id}_{(z_i)}\otimes\tilde{\phi}^{(i)}_{j,l})\Big)
  (z_i\otimes \epsilon^{(i)}_{j,l})
  \\
  &=
  (w^{(i)}_j)^{r^{(i)}_j+l-k}\epsilon^{(i)}_{j,k}
  \hspace{30pt}  (\text{when $l<k$}).
\end{align*}
By the construction, the diagrams
\[
 \begin{CD}
 \tilde{V}^{(i)}_{j,k}\otimes A[w^{(i)}_j]/((w^{(i)}_j)^{m_ir^{(i)}_j-r^{(i)}_j+1})
 @>\tilde{\pi}^{(i)}_{j,k}>> \tilde{L}^{(i)}_{j,k} \\
 @VVV  @VV\tilde{\phi}^{(i)}_{j,k}V \\
 \tilde{V}^{(i)}_{j,k-1}\otimes A[w^{(i)}_j]/((w^{(i)}_j)^{m_ir^{(i)}_j-r^{(i)}_j+1})
 @>\tilde{\pi}^{(i)}_{j,k-1}>> \tilde{L}^{(i)}_{j,k-1} \\
 \end{CD}
\]
are commutative
for $k=1,\ldots,r^{(i)}_j-1$ and the diagram
\[
 \begin{CD}
  (z_i)/(z_i^{m_i+1})\otimes V^{(i)}_{j,0} @>1\otimes\pi^{(i)}_{j,0}>> (z_i)/(z_i^{m_i+1})\otimes L^{(i)}_{j,0} \\
  @VVV  @VV\tilde{\phi}^{(i)}_{j,r^{(i)}_j}V \\
  \tilde{V}^{(i)}_{j,r^{(i)}_j-1}\otimes A[w^{(i)}_j]/((w^{(i)}_j)^{m_ir^{(i)}_j-r^{(i)}_j+1})
 @>\tilde{\pi}^{(i)}_{j,r^{(i)}_j-1}>> \tilde{L}^{(i)}_{j,r^{(i)}_j-1} 
 \end{CD}
\]
also commutes.
Define a homomorphism
$\tilde{l}^{(i)}_j/\tilde{l}^{(i)}_{j+1}
\longrightarrow
\tilde{l}^{(i)}_j/\tilde{l}^{(i)}_{j+1} \otimes \Omega^1_C(D)$
whose representation matrix with respect to the basis
$\tilde{f}^{(i)}_{j,0},\ldots,\tilde{f}^{(i)}_{j,r^{(i)}_j-1}$ is
\[
 \begin{pmatrix}
  \tilde{\nu}^{(i)}_{j,0}(z_i) & z_i\tilde{\nu}^{(i)}_{j,r^{(i)}-1}(z_i) & \cdots & z_i\tilde{\nu}^{(i)}_{j,1}(z_i) \\
  \tilde{\nu}'^{(i)}_{j,1}(z_i) & \tilde{\nu}^{(i)}_{j,0}(z_i)+\frac{1}{r^{(i)}_j}\frac{dz_i}{z_i}
   & \cdots & z_i\tilde{\nu}^{(i)}_{j,2}(z_i) \\
  \vdots & \vdots & \ddots & \vdots \\
  \tilde{\nu}'^{(i)}_{j,r^{(i)}_j-1}(z_i) & \tilde{\nu}'^{(i)}_{j,r^{(i)}_j-2}(z_i)
   & \cdots & \tilde{\nu}^{(i)}_{j,0}(z_i)+\frac{r^{(i)}_j-1}{r^{(i)}_j}\frac{dz_i}{z_i}
 \end{pmatrix},
\]
where $\tilde{\nu}'^{(i)}_{j,k}(z_i)$ is a lift of $\nu'^{(i)}_{j,k}(z_i)$ satisfying
$\tilde{\nu}'^{(i)}_{j,k}(z_i)\equiv \tilde{\nu}^{(i)}_{j,k}(z_i) \pmod {z_i^{m_i-1}dz_i/z_i^{m_i}}$.
Then the diagrams
\[
\begin{CD}
 \tilde{V}^{(i)}_{j,k}\otimes A[w^{(i)}_j]/((w^{(i)}_j)^{m_ir^{(i)}_j-r^{(i)}_j+1})
 @>\tilde{\pi}^{(i)}_{j,k}>> \tilde{L}^{(i)}_{j,k} \\
 @V\tilde{\nabla}^{(i)}_j|_{\tilde{V}^{(i)}_{j,k}}\otimes\mathrm{id} VV
 @VV\tilde{\nu}^{(i)}_j(w^{(i)}_j)+\frac{k\,dz_i}{r^{(i)}_j z_i} V \\
 \tilde{V}^{(i)}_{j,k}\otimes\Omega^1_C(D)
 \otimes A[w^{(i)}_j]/((w^{(i)}_j)^{m_ir^{(i)}_j-r^{(i)}_j+1})
 @>\tilde{\pi}^{(i)}_{j,k}\otimes 1 >> \tilde{L}^{(i)}_{j,k}\otimes\Omega^1_C(D) 
 \end{CD}
\]
are commutative for $k=0,\ldots,r^{(i)}_j-1$.

Choose an isomorphism
$\bigwedge^r E_{\alpha}\xrightarrow[\sim]{\sigma_{\alpha}}L|_{U_{\alpha}\otimes A}$
such that $\sigma_{\alpha}\otimes A/I$
is the given isomorphism
$\bigwedge^rE|_{U_{\alpha}\otimes A/I}\xrightarrow{\sim}L|_{U_{\alpha}\otimes A/I}$.
We can give a connection
$\nabla_{\alpha}\colon E_{\alpha}\longrightarrow E_{\alpha}\otimes\Omega^1_C(D)$
which induces $\tilde{\nabla}^{(i)}_j$ on $\tilde{l}^{(i)}_j/\tilde{l}^{(i)}_{j+1}$ for each $i,j$.
Then we have
\[
 \Tr\left( \nabla_{\alpha}|_{m_it_i\times\Spec A} \right)
 =\sum_{j=0}^{s_i-1} \left( r^{(i)}_j\tilde{\nu}^{(i)}_{j,0}(z_i)+(r^{(i)}_j-1)dz_i/z_i\right).
\]
So, after adjusting the diagonal entries of $\nabla_{\alpha}$, we may assume that
the connection $\det(E_{\alpha},\nabla_{\alpha})$ induced by $\nabla_{\alpha}$
on the determinant bundle
is transformed to the connection $\nabla_L|_{U_{\alpha}\times\Spec A}$
via the isomorphism
$\sigma_{\alpha}$.
Thus we obtain a local parabolic connection
$(E_{\alpha},\nabla_{\alpha},\{\tilde{l}^{(i)}_j,\tilde{V}^{(i)}_{j,k},\tilde{L}^{(i)}_{j,k},
\tilde{\pi}^{(i)}_{j,k},\tilde{\phi}^{(i)}_{j,k}\})$ 
of generic ramified type with the exponent $\tilde{\bnu}$
on $U_{\alpha}\times\Spec A$,
which is a lift of the given parabolic connection
$(E,\nabla,\{l^{(i)}_j,V^{(i)}_{j,k},L^{(i)}_{j,k},\pi^{(i)}_{j,k},\phi^{(i)}_{j,k}\})
\big|_{U_{\alpha}\otimes A/I}$
of generic ramified type on $U_{\alpha}\otimes A/I$.

If $t_i\notin U_{\alpha}$ for any $i$,
then we can easily give a local parabolic connection
on $U_{\alpha}\otimes A$
which is a lift of $(E,\nabla,\{l^{(i)}_j,V^{(i)}_{j,k},L^{(i)}_{j,k},\pi^{(i)}_{j,k},\phi^{(i)}_{j,k}\})|_{U_{\alpha}\otimes A/I}$.
Note that the data
$\{\tilde{l}^{(i)}_j,\tilde{V}^{(i)}_{j,k},\tilde{L}^{(i)}_{j,k},\tilde{\pi}^{(i)}_{j,k},\tilde{\phi}^{(i)}_{j,k}\}$ is nothing 
in this case.

Recall the complex ${\mathcal F}_{M'}^{\bullet}$
defined in (\ref {equation: definition of tangent complex})
and consider its restriction
${\mathcal F}^{\bullet}_x:={\mathcal F}^{\bullet}_{M'}|_{C\times\{x\}}$,
to a point $x$ of $M'$ lying over the given point
$\Spec A/\mathfrak{m}\longrightarrow{\mathcal M}^{\balpha}_{C,D}((r^{(i)}_j),a)$.
Define a complex ${\mathcal F}^{\bullet}_{\mathfrak{sl},x}$ by
\begin{align*}
 {\mathcal F}^0_{\mathfrak{sl},x}
 &:=
 \left\{ u\in {\mathcal F}^0_x \,
  \middle| \,  \Tr(u)=0 \right\} \\
 {\mathcal F}^1_{\mathfrak{sl},x}
 &:=
 \left\{ v\in {\mathcal F}^1_x \,
  \middle| \,  \Tr(v)=0 \right\}
  \\
 \nabla_{{\mathcal F}^{\bullet}_{\mathfrak{sl},x}}
 &:=
 \nabla_{{\mathcal F}^{\bullet}_x}|_{{\mathcal F}^0_{\mathfrak{sl},x}}
 \colon {\mathcal F}^0_{\mathfrak{sl},x}
 \longrightarrow {\mathcal F}^1_{\mathfrak{sl},x}.
\end{align*}
For $U_{\alpha\beta}:=U_{\alpha}\cap U_{\beta}$,
we take a lift
$\theta_{\beta\alpha}:E_{\alpha}|_{U_{\alpha\beta}\otimes A}\xrightarrow{\sim}E_{\beta}|_{U_{\alpha\beta}\otimes A}$
of the canonical isomorphism
$E_{\alpha}\otimes A/I|_{U_{\alpha\beta}\otimes A/I}
\xrightarrow[\sim]{\tau_{\alpha}}E|_{U_{\alpha\beta}\otimes A/I}
\xrightarrow[\sim]{\tau_{\beta}^{-1}}E_{\beta}\otimes A/I|_{U_{\alpha\beta}\otimes A/I}$
such that
$\sigma_{\beta}\circ\det(\theta_{\beta\alpha})=\sigma_{\alpha}$.
We put
\[
 u_{\alpha\beta\gamma} 
 :=
 \tau_{\alpha}(\theta_{\gamma\alpha}^{-1}\theta_{\gamma\beta}\theta_{\beta\alpha}
 -\mathrm{id})\tau_{\alpha}^{-1},
 \quad
 v_{\alpha\beta} 
 :=
 \tau_{\alpha}(\nabla_{\alpha}
 -\theta_{\beta\alpha}^{-1}\nabla_{\beta}\theta_{\beta\alpha})\tau_{\alpha}^{-1}.
\]
Then we have 
$\{u_{\alpha\beta\gamma}\}\in C^2(\{U_{\alpha}\},{\mathcal F}^0_{\mathfrak{sl},x}\otimes I)$
and $\{v_{\alpha\beta}\}\in C^1(\{U_{\alpha}\},{\mathcal F}^1_{\mathfrak{sl},x}\otimes I)$.
We can check that
$\{u_{\beta\gamma\delta}-u_{\alpha\gamma\delta}+u_{\alpha\beta\delta}
-u_{\alpha\beta\gamma}\}=0$ and
$\nabla_{{\mathcal F}^{\bullet}_{\mathfrak{sl},x}}(\{u_{\alpha\beta\gamma}\})
=-\{ v_{\beta\gamma}-v_{\alpha\gamma}+v_{\alpha\beta}\}$.
So we can define an element
\[
 \omega(E,\nabla,\{l^{(i)}_j,V^{(i)}_{j,k},L^{(i)}_{j,k},\pi^{(i)}_{j,k},\phi^{(i)}_{j,k}\})
 :=[(\{u_{\alpha\beta\gamma}\},\{v_{\alpha\beta}\})]
 \in \mathbf{H}^2({\mathcal F}^{\bullet}_{\mathfrak{sl},x}\otimes I).
\]
in the hyper cohomology group
$\mathbf{H}^2({\mathcal F}^{\bullet}_{\mathfrak{sl},x}\otimes I)$.
Considering a gluing condition of the local parabolic connections
$(E_{\alpha},\nabla_{\alpha},\{\tilde{l}^{(i)}_j,\tilde{V}^{(i)}_{j,k},\tilde{L}^{(i)}_{j,k},
\tilde{\pi}^{(i)}_{j,k},\tilde{\phi}^{(i)}_{j,k}\})$,
we can see that the vanishing of the element
$\omega(E,\nabla,\{l^{(i)}_j,V^{(i)}_{j,k},L^{(i)}_{j,k},\pi^{(i)}_{j,k},\phi^{(i)}_{j,k}\})$
in $\mathbf{H}^2({\mathcal F}^{\bullet}_{\mathfrak{sl},x}\otimes I)$
is equivalent to the existence of an $A$-valued point
$(\tilde{E},\tilde{\nabla},\{\tilde{l}^{(i)}_j,\tilde{V}^{(i)}_{j,k},\tilde{L}^{(i)}_{j,k},
\tilde{\pi}^{(i)}_{j,k},\tilde{\phi}^{(i)}_{j,k}\})$
of ${\mathcal M}^{\balpha}_{C,D}((r^{(i)}_j),a)$ such that
\[
 (\tilde{E},\tilde{\nabla},\{\tilde{l}^{(i)}_j,\tilde{V}^{(i)}_{j,k},\tilde{L}^{(i)}_{j,k},
 \tilde{\pi}^{(i)}_{j,k},\tilde{\phi}^{(i)}_{j,k}\})\otimes A/I
 \cong
 (E,\nabla,\{l^{(i)}_j,V^{(i)}_{j,k},L^{(i)}_{j,k},\pi^{(i)}_{j,k},\phi^{(i)}_{j,k}\}).
\]
From the hyper cohomology spectral sequence
$H^q(\bar{\mathcal F}^p_0)\Rightarrow \mathbf{H}^{p+q}(\bar{\mathcal F}^{\bullet}_0)$,
we have
\begin{align*}
 \mathbf{H}^2({\mathcal F}^{\bullet}_{\mathfrak{sl},x}) &\cong
 \coker( H^1({\mathcal F}^0_{\mathfrak{sl},x})\longrightarrow 
 H^1({\mathcal F}^1_{\mathfrak{sl},x})) \\
 &\cong \coker \left( H^0(({\mathcal F}^0_{\mathfrak{sl},x})^{\vee}\otimes\Omega^1_C)^{\vee}
 \longrightarrow H^0(({\mathcal F}^1_{\mathfrak{sl},x})^{\vee}\otimes\Omega^1_C)^{\vee}\right) \\
 &\cong \ker \left( H^0(({\mathcal F}^1_{\mathfrak{sl},x})^{\vee}\otimes\Omega^1_C)
 \xrightarrow{-H^0(\nabla^{\dag})} 
 H^0(({\mathcal F}^0_{\mathfrak{sl},x})^{\vee}\otimes\Omega^1_C)\right)^{\vee}.
\end{align*}
If we put
$(\bar{E},\bar{\nabla},\{\bar{l}^{(i)}_j,\bar{V}^{(i)}_{j,k},
\bar{L}^{(i)}_{j,k},\bar{\pi}^{(i)}_{j,k},\bar{\phi}^{(i)}_{j,k}\})
:=
(E,\nabla,\{l^{(i)}_j,V^{(i)}_{j,k},L^{(i)}_{j,k},\pi^{(i)}_{j,k},\phi^{(i)}_{j,k}\})
\otimes A/\mathfrak{m}$,
then we can check the equalities
\begin{align*}
 ({\mathcal F}^1_{\mathfrak{sl},x})^{\vee}\otimes\Omega^1_C &=
 \left\{
  u\in{\mathcal End}(\bar{E}) \left|
  \begin{array}{l}
   \text{$\Tr(u)=0$, $u|_{m_it_i}(\bar{l}^{(i)}_j)\subset \bar{l}^{(i)}_j$ for any $i,j$ and} \\
   \text{for the induced
   $u^{(i)}_j \colon\bar{l}^{(i)}_j/\bar{l}^{(i)}_{j+1}\longrightarrow\bar{l}^{(i)}_j/\bar{l}^{(i)}_{j+1}$,} \\
   \text{$u^{(i)}_j(V^{(i)}_{j,k})\subset V^{(i)}_{j,k}$ for any $k$} \\
  \end{array}
 \right.
 \right\} \\
 ({\mathcal F}^0_{\mathfrak{sl},x})^{\vee}\otimes\Omega^1_C &=
 \left\{
  v\in{\mathcal End}(\bar{E})\otimes\Omega^1_C(D)
  \left|
  \begin{array}{l}
   \text{$\Tr(v)=0$ and $\Tr(v\circ u)\in\Omega^1_C$} \\
   \text{for any $u\in {\mathcal F}^0_{\mathfrak{sl},x}$}
  \end{array}
  \right.
 \right\}, \\
 & \nabla^{\dag}\colon
 ({\mathcal F}^1_{\mathfrak{sl},x})^{\vee}\otimes\Omega^1_C
 \ni u\mapsto \nabla u-u\nabla\in
 ({\mathcal F}^0_{\mathfrak{sl},x})^{\vee}\otimes\Omega^1_C.
\end{align*}
We will check that 
$\nabla^{\dag}\colon ({\mathcal F}^1_{\mathfrak{sl},x})^{\vee}\otimes\Omega^1_C
\longrightarrow ({\mathcal F}^0_{\mathfrak{sl},x})^{\vee}\otimes\Omega^1_C$
is indeed defined.
Take $u\in({\mathcal F}^1_{\mathfrak{sl},x})^{\vee}\otimes\Omega^1_C$
and $u'\in {\mathcal F}^0_{\mathfrak{sl},x}$.
Then there are
homomorphisms 
$\beta^{(i)}_{j,k}\colon L^{(i)}_{j,k}\longrightarrow L^{(i)}_{j,k}$
of ${\mathcal O}_{m_it_i}$-modules
and elements
$\gamma^{(i)}_{j,k}\in\mathbb{C}[w^{(i)}_j]/((w^{(i)}_j)^{m_ir^{(i)}_j-r^{(i)}_j+1})$
satisfying the commutative diagrams
\[
 \begin{CD}
  V^{(i)}_{j,k} @> u^{(i)}_j >> V^{(i)}_{j,k} \\
  @V\pi^{(i)}_{j,k}|_{V^{(i)}_{j,k}} VV  @VV \pi^{(i)}_{j,k}|_{V^{(i)}_{j,k}} V \\
  L^{(i)}_{j,k} @> \beta^{(i)}_{j,k} >> L^{(i)}_{j,k}
 \end{CD}
 \hspace{80pt}
 \begin{CD}
  V^{(i)}_{j,k} @> u'^{(i)}_j >> V^{(i)}_{j,k} \\
  @V\pi^{(i)}_{j,k}|_{V^{(i)}_{j,k}} VV  @VV \pi^{(i)}_{j,k}|_{V^{(i)}_{j,k}}V \\
  L^{(i)}_{j,k} @> \gamma^{(i)}_{j,k} >> L^{(i)}_{j,k},
 \end{CD}
\]
where $u^{(i)}_j \colon V^{(i)}_{j,0}\longrightarrow V^{(i)}_{j,0}$
and $u'^{(i)}_j\colon V^{(i)}_{j,0}\longrightarrow V^{(i)}_{j,0}$
are homomorphisms induced by $u|_{m_it_i}$ and $u'|_{m_it_i}$, respectively.
Then we have
\begin{align*}
 &
 \pi^{(i)}_{j,k}((u'^{(i)}_j\nabla^{(i)}_j-\nabla^{(i)}_j u'^{(i)}_j) u^{(i)}_j|_{V^{(i)}_{j,k}})
 \\
 &=
 \left( \gamma^{(i)}_{j,k} \big( \nu^{(i)}_j(w^{(i)}_j)+kdz_i/r^{(i)}_jz_i \big)
 -\big( \nu^{(i)}_j(w^{(i)}_j)+kdz_i/r^{(i)}_jz_i \big) \gamma^{(i)}_{j,k} \right)
 \beta^{(i)}_{j,k}\pi^{(i)}_{j,k}=0
\end{align*}
for $k=0,\ldots,r^{(i)}_j-1$.
So we can see that $\Tr((u'^{(i)}_j\nabla^{(i)}_j-\nabla^{(i)}_j u'^{(i)}_j) u^{(i)}_j)=0$ and
\begin{align*}
 \Tr( u'^{(i)}_j (\nabla^{(i)}_j u^{(i)}_j - u^{(i)}_j \nabla^{(i)}_j ) ) &=
 \Tr( u'^{(i)}_j (\nabla^{(i)}_j u^{(i)}_j - u^{(i)}_j \nabla^{(i)}_j ))
 -\Tr((u'^{(i)}_j\nabla^{(i)}_j-\nabla^{(i)}_j u'^{(i)}_j) u^{(i)}_j) \\
 &=-\Tr( (u'^{(i)}_j u^{(i)}_j) \nabla^{(i)}_j ) + \Tr( \nabla^{(i)}_j (u'^{(i)}_j u^{(i)}_j) )=0
\end{align*}
for any $i,j$. Thus we have $\Tr(u'\circ(\nabla u-u\nabla))\in\Omega^1_C$,
$\nabla u-u\nabla\in({\mathcal F}^0_{\mathfrak{sl},x})^{\vee}\otimes\Omega^1_C$
and the morphism
$\nabla^{\dag}\colon ({\mathcal F}^1_{\mathfrak{sl},x})^{\vee}\otimes\Omega^1_C(D)
\longrightarrow ({\mathcal F}^0_{\mathfrak{sl},x})^{\vee}\otimes\Omega^1_C$
can be defined certainly.

If we take
$u\in \ker( H^0(({\mathcal F}^1_{\mathfrak{sl},x})^{\vee}\otimes\Omega^1_C)
 \xrightarrow{-\nabla^{\dag}} 
 H^0(({\mathcal F}^0_{\mathfrak{sl},x})^{\vee}\otimes\Omega^1_C)$,
then $u:E\longrightarrow E$ is a homomorphism satisfying
$\nabla\circ u=(u\otimes\mathrm{id})\circ\nabla$ and
$u|_{m_it_i}(l^{(i)}_j)\subset l^{(i)}_j$.
Since $(E,\nabla,\{l^{(i)}_j\})$ is $\balpha$-stable,
we have $u=c\cdot\mathrm{id}$ for some $c\in A/\mathfrak{m}$.
So $\Tr(u)=0$ implies $c=0$ and $u=0$.
Thus
\[
 \ker( H^0(({\mathcal F}^1_{\mathfrak{sl},x})^{\vee}\otimes\Omega^1_C)
 \xrightarrow{-\nabla_{{\mathcal F}^{\bullet}_{\mathfrak{sl},x}}} 
 H^0(({\mathcal F}^0_{\mathfrak{sl},x})^{\vee}\otimes\Omega^1_C))=0
\]
and $\mathbf{H}^2({\mathcal F}^{\bullet}_{\mathfrak{sl},x})=0$.
In particular, we have
$\omega(E,\nabla,\{l^{(i)}_j,V^{(i)}_{j,k},L^{(i)}_{j,k},\pi^{(i)}_{j,k},\phi^{(i)}_{j,k}\})=0$,
which means that a parabolic connection
$(E,\nabla,\{l^{(i)}_j,V^{(i)}_{j,k},L^{(i)}_{j,k},\pi^{(i)}_{j,k},\phi^{(i)}_{j,k}\})$
of generic ramified type over $A/I$
can be lifted to an $A$-valued point of
${\mathcal M}^{\balpha}_{C,D}((r^{(i)}_j),a)_{\tilde{\bnu}}$.
Hence $\det$ is a smooth morphism.
\end{proof}

For the proof of Theorem \ref{thm:generic-smooth},
it is sufficient to prove the following proposition.

\begin{proposition}
 For any point $x\in M^{\balpha}_{C,D}((r^{(i)}_j),a)_{\bnu}(\mathbb{C})$,
 the dimension of the tangent space
 $T_{M^{\balpha}_{C,D}((r^{(i)}_j),a)_{\bnu}}(x)$ 
 of $M^{\balpha}_{C,D}((r^{(i)}_j),a)_{\bnu}$ at $x$ is
\[
 2r^2(g-1)+2+\sum_{i=1}^n(r^2-r)m_i.
\]
\end{proposition}

\begin{proof}
Let $(E,\nabla,\{l^{(i)}_j,V^{(i)}_{j,k},L^{(i)}_{j,k},\pi^{(i)}_{j,k},\phi^{(i)}_{j,k}\})$
be the parabolic connection of generic ramified type with the exponent $\bnu$
which corresponds to the point $x$.
Recall the complex ${\mathcal F}^{\bullet}_{M'}$
defined in (\ref {equation: definition of tangent complex})
and consider its restriction
${\mathcal F}^{\bullet}_x
={\mathcal F}^{\bullet}_{M'}\big|_{C\times x}$,
where we denote a point of $M'$ lying over $x$ by the same symbol.
We will prove that the tangent space
$T_{M^{\balpha}_{C,D}((r^{(i)}_j),a)_{\bnu}}(x)$ of $M^{\balpha}_{C,D}((r^{(i)}_j),a)_{\bnu}$ at $x$
is isomorphic to the hyper cohomology
$\mathbf{H}^1({\mathcal F}^{\bullet}_x)$.

Take a tangent vector
$v\in T_{M^{\balpha}_{C,D}((r^{(i)}_j),a)_{\bnu}}(x)$
which corresponds to a member
\[
 (E^v,\nabla^v,\{(l^v)^{(i)}_j,(V^v)^{(i)}_{j,k},(L^v)^{(i)}_{j,k},(\pi^v)^{(i)}_{j,k},(\phi^v)^{(i)}_{j,k}\})
 \in M^{\balpha}_{C,D}((r^{(i)}_j),a)(\mathbb{C}[\epsilon])
\]
such that
\[
 (E^v,\nabla^v,\{(l^v)^{(i)}_j,(V^v)^{(i)}_{j,k},(L^v)^{(i)}_{j,k},(\pi^v)^{(i)}_{j,k},(\phi^v)^{(i)}_{j,k}\})
 \otimes\mathbb{C}[\epsilon]/(\epsilon)
 \cong(E,\nabla,\{l^{(i)}_j,V^{(i)}_{j,k},L^{(i)}_{j,k},\pi^{(i)}_{j,k},\phi^{(i)}_{j,k}\}).
\]
We take an affine open covering
$\{U_{\alpha}\}$ of $C$ as in the proof of Proposition \ref{prop:det-smooth}.
Take a lift
\[
 \varphi_{\alpha}\colon E^v|_{U_{\alpha}\otimes\mathbb{C}[\epsilon]}
 \xrightarrow{\sim} E\otimes\mathbb{C}[\epsilon]|_{U_{\alpha}\otimes\mathbb{C}[\epsilon]}
\]
of the given isomorphism
$E^v\otimes\mathbb{C}[\epsilon]/(\epsilon)|_{U_{\alpha}}
\xrightarrow{\sim} E|_{U_{\alpha}}$
such that the restriction
$\varphi_{\alpha}|_{m_it_i\otimes\mathbb{C}[\epsilon]}$ sends the data
$\{(l^v)^{(i)}_j,(V^v)^{(i)}_{j,k},(L^v)^{(i)}_{j,k},(\pi^v)^{(i)}_{j,k},(\phi^v)^{(i)}_{j,k}\}$
to $\{l^{(i)}_j,V^{(i)}_{j,k},L^{(i)}_{j,k},\pi^{(i)}_{j,k},\phi^{(i)}_{j,k}\}\otimes\mathbb{C}[\epsilon]$
if $t_i\in U_{\alpha}$.
We put
\begin{align*}
 u_{\alpha\beta} &:=
 \varphi_{\alpha}\circ\varphi_{\beta}^{-1}-\mathrm{id}_{E|_{U_{\alpha\beta}\otimes\mathbb{C}[\epsilon]}}, \\
 v_{\alpha} &:= (\varphi_{\alpha}\otimes\mathrm{id})\circ\nabla^v\circ\varphi_{\alpha}^{-1}
 -\nabla|_{U_{\alpha}}\otimes\mathrm{id}_{\mathbb{C}[\epsilon]}.
\end{align*}
Then we have
$\{u_{\alpha\beta}\}\in C^1(\{U_{\alpha}\},(\epsilon)\otimes{\mathcal F}^0_x)$,
$\{v_{\alpha}\}\in C^0(\{U_{\alpha}\},(\epsilon)\otimes{\mathcal F}^1_x)$ and
\[
 d\{u_{\alpha\beta}\}=\{ u_{\beta\gamma}-u_{\alpha\gamma}+u_{\alpha\beta}\}=0, \quad
 \nabla_{{\mathcal F}^{\bullet}_x}\{u_{\alpha\beta}\}=\{v_{\beta}-v_{\alpha}\}=d\{v_{\alpha}\}.
\]
So $[(\{u_{\alpha\beta}\},\{v_{\alpha}\})]$ gives an element
$\Phi(v)$ of $\mathbf{H}^1({\mathcal F}^{\bullet}_x)$.
We can check that the correspondence
$v\mapsto \Phi(v)$ gives an isomorphism
\[
 \Phi \colon T_{M^{\balpha}_{C,D}((r^{(i)}_j),a)_{\bnu}}(x)
 \xrightarrow{\sim} \mathbf{H}^1({\mathcal F}^{\bullet}_x).
\]
From the hyper cohomology spectral sequence
$H^q({\mathcal F}^p_x)\Rightarrow \mathbf{H}^{p+q}({\mathcal F}^{\bullet}_x)$,
we obtain an exact sequence
\[
 0\longrightarrow \mathbb{C} \longrightarrow H^0({\mathcal F}^0_x) \longrightarrow
 H^0({\mathcal F}^1_x) \longrightarrow \mathbf{H}^1({\mathcal F}_x^{\bullet})
 \longrightarrow H^1({\mathcal F}^0_x) \longrightarrow H^1({\mathcal F}^1_x) \longrightarrow
 \mathbb{C} \longrightarrow 0.
\]
So we have
\begin{align*}
 \dim \mathbf{H}^1({\mathcal F}^{\bullet}_x)
 &= \dim H^0({\mathcal F}^1_x) + \dim H^1({\mathcal F}^0_x)
 -\dim H^0({\mathcal F}^0_x) - \dim H^1({\mathcal F}^1_x) + 2\dim_{\mathbb{C}}\mathbb{C} \\
 &= \chi({\mathcal F}^1_x)-\chi({\mathcal F}^0_x) +2.
\end{align*}
We can calculate $\chi({\mathcal F}^1_x)$ and $\chi({\mathcal F}^0_x)$
by its definition as follows:
\begin{align*}
 \chi({\mathcal F}^1_x) 
 &= 
 (1-g)\rank{\mathcal F}^1_x +\deg({\mathcal F}^1_x) \\
 &= r^2(1-g) +r^2(2g-2) +\sum_{i=1}^n\sum_{j'<j} m_i r^{(i)}_{j'} r^{(i)}_j 
 +\sum_{i=1}^n\sum_{j=0}^{s_i-1}\sum_{k=0}^{r^{(i)}_j-1}(r^{(i)}_j-1-k) \\
 \chi({\mathcal F}^0_x) 
 &= 
 (1-g) \rank{\mathcal F}^0_x+ \deg({\mathcal F}^0_x) \\
 &= r^2(1-g) - \sum_{i=1}^n\sum_{j'\leq j} m_i r^{(i)}_{j'}r^{(i)}_j 
 +\sum_{i=1}^n\sum_{j=0}^{s_i-1}(m_ir^{(i)}_j-r^{(i)}_j+1+\sum_{k=0}^{r^{(i)}_j-1}1
 +\sum_{k=0}^{r^{(i)}_j-1}(r^{(i)}_j-1-k)) \\
 &= r^2(1-g) - \sum_{i=1}^n\sum_{j'\leq j} m_i r^{(i)}_{j'}r^{(i)}_j + \sum_{i=1}^n\sum_{j=0}^{s_i-1}m_ir^{(i)}_j
  +\sum_{i=1}^n\sum_{j=0}^{s_i-1}\sum_{k=0}^{r^{(i)}_j-1}(r^{(i)}_j-1-k).
\end{align*}
So we have
\begin{align*}
 \dim\mathbf{H}^1({\mathcal F}^{\bullet}_x)
 &=
 \chi({\mathcal F}^1_x)-\chi({\mathcal F}^0_x)+2 \\
 &=2r^2(g-1)+2+\sum_{i=1}^n\sum_{j'<j} 2m_i r^{(i)}_{j'}r^{(i)}_j
 +\sum_{i=1}^n\sum_{j=0}^{s_i-1} m_i (r^{(i)}_j)^2
 -\sum_{i=1}^n\sum_{j=0}^{s_i-1}m_ir^{(i)}_j \\
 &=2r^2(g-1)+2+\sum_{i=1}^nm_ir(r-1)
\end{align*}
and the result follows. 
\end{proof}


\section{Symplectic form on the moduli space of parabolic connections of generic ramified type
for generic exponent}\label{section:symplectic-form}

We will give in this section a symplectic form on the moduli space of
$\balpha$-stable parabolic connections of generic ramified type
when the exponent $\bnu$ is generic.
For this we enlarge the moduli space
to the moduli space of simple
parabolic connections of generic ramified type, which is a non-separated algebraic space.

\begin{definition}\label{def:simple-connection}\rm
We define a functor
${\mathcal M}^{\rm spl}_{C,D}((r^{(i)}_j),a) \colon
(\sch/N((r^{(i)}_j),a))^o\longrightarrow (\sets)$ by
\[
 {\mathcal M}^{\rm spl}_{C,D}((r^{(i)}_j),a)(S)
 :=
 \left\{ (E,\nabla,\{l^{(i)}_j,V^{(i)}_{j,k},L^{(i)}_{j,k},\pi^{(i)}_{j,k},\phi^{(i)}_{j,k}\}) \right\}/\sim
\]
for a noetherian scheme $S$ over $\mathbb{C}$, where
$(E,\nabla,\{l^{(i)}_j,V^{(i)}_{j,k},L^{(i)}_{j,k},\pi^{(i)}_{j,k},\phi^{(i)}_{j,k}\})$
is the same as in Definition \ref {def:moduli-functor} 
except for replacing the condition (v) by
\begin{itemize}
\item[(v')]
for any geometric point $s$ of $S$,
the equality
$\mathrm{End}((E,\nabla,\{l^{(i)}_j\})|_{C\times s})
=\mathbb{C}\mathrm{id}_E$ holds,
where we define
$\mathrm{End}((E,\nabla,\{l^{(i)}_j\})|_{C\times s})$
as the set of the endomorphisms 
$u\colon E\otimes k(s)\longrightarrow E\otimes k(s)$
satisfying
$(\nabla\otimes k(s))\circ u=(u\otimes 1)\circ(\nabla\otimes k(s))$
and
$u|_{m_it_i\times s}(l^{(i)}_j\otimes k(s))\subset l^{(i)}_j\otimes k(s)$ for any $i,j$.
\end{itemize}
Here the relation $\sim$ is the same as that in Definition \ref{def:moduli-functor}.
\end{definition}

The next proposition is much easier than the proof of Theorem \ref{thm-existence-moduli-generic}.

\begin{proposition}\label{prop:algebraic-space-moduli}
The \'etale sheafification of ${\mathcal M}^{\rm spl}_{C,D}((r^{(i)}_j),a)$ is represented by
an algebraic space $M^{\rm spl}_{C,D}((r^{(i)}_j),a)$, locally of finite type over $N((r^{(i)}_j),a)$.
Moreover, $M^{\rm spl}_{C,D}((r^{(i)}_j),a)$ is smooth over $N((r^{(i)}_j),a)$ 
whose fiber has dimension
$2r^2(g-1)+2+\sum_{i=1}^n(r^2-r)m_i$ if it is non-empty.
\end{proposition}

\begin{proof}
For each integer $m$, we consider the subfunctor 
${\mathcal M}^{\rm spl,m}_{C,D}((r^{(i)}_j),a)\subset
{\mathcal M}^{\rm spl}_{C,D}((r^{(i)}_j),a)$
such that
$(E,\nabla,\{l^{(i)}_j,V^{(i)}_{j,k},L^{(i)}_{j,k},\pi^{(i)}_{j,k},\phi^{(i)}_{j,k}\})\in
{\mathcal M}^{\rm spl}_{C,D}((r^{(i)}_j),a)_{\bnu}(S)$
lies in ${\mathcal M}^{\rm spl,m}_{C,D}((r^{(i)}_j),a)(S)$
if $E|_{C\times s}$ is $m$-regular for any $s\in S$.
If we put $N:=h^0(E(m))$, then all $m$-regular vector bundles of rank $r$ and degree $a$
can be parametrized by a Zariski open set
$Q$ of the Quot-scheme $\Quot_{{\mathcal O}_C^{\oplus N}(-m)}$.
Take a universal quotient bundle
${\mathcal O}_{C\times Q}^{\oplus N}(-m)\longrightarrow \tilde{E}$.
We take the ${\mathcal O}_C$-bimodule $\Lambda^1_D$ given in
\cite{IIS-1}, section 5.1.
Then we can construct a quasi-projective scheme $R$ over $Q$
such that there is a functorial isomorphism
$R(S)\cong\Hom(\Lambda^1_D\otimes\tilde{E}_S,\tilde{E}_S)$
for noetherian schemes $S$ over $R$.
The universal family over $R$ corresponds to
$(\tilde{\phi}\colon\tilde{E}_R\longrightarrow\tilde{E}_R,
\tilde{\nabla}\colon\tilde{E}_R\longrightarrow\tilde{E}_R\otimes\Omega^1_C(D))$.
If $R'$ is the subscheme of $R$ where $\tilde{\phi}$ becomes identity,
then $\tilde{\nabla}_{R'}\colon\tilde{E}_{R'}\longrightarrow\tilde{E}_{R'}\otimes\Omega^1_C(D)$
becomes a relative connection.
We can construct a locally closed subscheme $F_{R'}$
of a flag scheme over $R'$
parameterizing the parabolic structures
$(l^{(i)}_j)$ on $\tilde{E}_{R'}$ preserved by
$\nabla|_{(m_it_i)_{R'}}$.

By the same argument as in the proof of
Theorem \ref {thm-existence-moduli-generic}
we can construct a quasi-projective scheme $U$ over $F_{R'}$
parameterizing $\nu^{(i)}_j(w^{(i)}_j)$-ramified structures on
$(\tilde{l}^{(i)}_j/\tilde{l}^{(i)}_{j+1},\tilde{\nabla}^{(i)}_j)_{i,j}$,
where $(\tilde{l}^{(i)}_j)$ is a universal family of parabolic structure
on $\tilde{E}_{F_{R'}}$ and
$\tilde{\nabla}^{(i)}_j \colon  \tilde{l}^{(i)}_j/\tilde{l}^{(i)}_{j+1}
\longrightarrow (\tilde{l}^{(i)}_j/\tilde{l}^{(i)}_{j+1} \otimes\Omega^1_C(D)$
is the homomorphism induced by $\tilde{\nabla}$.
There is a canonical morphism
$U\longrightarrow {\mathcal M}^{\rm spl,m}_{C,D}((r^{(i)}_j),a)$
which is formally smooth by construction.
The action of $PGL_N(\mathbb{C})$ on $\Quot_{{\mathcal O}_C^{\oplus N}(-m)}$
canonically lifts to a free action on $U$.
We can see that the quotient $M^{\rm spl,m}_{C,D}((r^{(i)}_j),a):=U/PGL_N(\mathbb{C})$
as an algebraic space represents the \'etale sheafification of
${\mathcal M}^{\rm spl,m}_{C,D}((r^{(i)}_j),a)$.
There are canonical inclusions
$M^{\rm spl,m}_{C,D}((r^{(i)}_j),a)\hookrightarrow M^{\rm spl,m+1}_{C,D}((r^{(i)}_j),a)$
and
\[
 M^{\rm spl}_{C,D}((r^{(i)}_j),a):=
 \bigcup_{m\geq 0}M^{{\rm spl},m}_{C,D}((r^{(i)}_j),a)
\]
gives the desired moduli space.
Smoothness and the dimension counting follow from the same argument as that of the proof of
Theorem \ref{thm:generic-smooth}.
\end{proof}

We can take a quasi-finite \'etale covering
$M'_{\bnu}\longrightarrow M^{\rm spl}_{C,D}((r^{(i)}_j),a)_{\bnu}$
such that
there is a universal family
$(E_{M'_{\bnu}},\nabla_{M'_{\bnu}},\{l^{(i)}_{M'_{\bnu},j},V^{(i)}_{M'_{\bnu},j,k},
L^{(i)}_{M'_{\bnu},j,k},\pi^{(i)}_{M'_{\bnu},j,k},\phi^{(i)}_{M'_{\bnu},j,k}\})$
over $M'_{\bnu}$.
Define the complex ${\mathcal F}^{\bullet}_{M'_{\bnu}}$ in the same way as
the definition of ${\mathcal F}^{\bullet}_{M'}$
in (\ref {equation: definition of tangent complex}).
We can see that the tangent bundle
$T_{M'_{\bnu}}$
is isomorphic to $\mathbf{R}^1(\pi_{M'_{\bnu}})_*({\mathcal F}^{\bullet}_{M'_{\bnu}})$,
where
$\pi_{M'_{\bnu}} \colon C\times M'_{\bnu}\longrightarrow M'_{\bnu}$
is the second projection.
We consider the pairing
\begin{gather*}
 \omega_{M'_{\bnu}} \colon
 \mathbf{R}^1(\pi_{M'_{\bnu}})_*({\mathcal F}^{\bullet}_{M'_{\bnu}})
 \times
 \mathbf{R}^1(\pi_{M'_{\bnu}})_*({\mathcal F}^{\bullet}_{M'_{\bnu}})
 \longrightarrow
 \mathbf{R}^2(\pi_{M'_{\bnu}})_*(\Omega^{\bullet}_{C\times M'_{\bnu}/M'_{\bnu}})
 \cong{\mathcal O}_{M'_{\bnu}} \\
 ([\{u_{\alpha\beta}\}, \{v_{\alpha}\}], [\{u'_{\alpha\beta}\}, \{v'_{\alpha}\}])
 \mapsto [(\{\Tr(u_{\alpha\beta}\circ u'_{\beta\gamma})\},
 -\{\Tr(u_{\alpha\beta}\circ v'_{\beta}-v_{\alpha}\circ u'_{\alpha\beta})\}]
\end{gather*}
We can easily see that $\omega_{M'_{\bnu}}$ descends to a pairing
\begin{equation} \label {equation: symplectic form}
 \omega_{M^{\rm spl}_{C,D}((r^{(i)}_j),a)_{\bnu}}\colon
 T_{M^{\rm spl}_{C,D}((r^{(i)}_j),a)_{\bnu}} \times
 T_{M^{\rm spl}_{C,D}((r^{(i)}_j),a)_{\bnu}}
 \longrightarrow {\mathcal O}_{M^{\rm spl}_{C,D}((r^{(i)}_j),a)_{\bnu}}.
\end{equation}
Take a tangent vector
$v\in T_{M^{\rm spl}_{C,D}((r^{(i)}_j),a)_{\bnu}}(x)$
corresponding to a $\mathbb{C}[\epsilon]$-valued point
\[
 (E^v,\nabla^v,\{(l^v)^{(i)}_j,(V^v)^{(i)}_{j,k},(L^v)^{(i)}_{j,k},(\pi^v)^{(i)}_{j,k},(\phi^v)^{(i)}_{j,k}\})
\]
of $M^{\rm spl}_{C,D}((r^{(i)}_j),a)_{\bnu}$.
We can see by an easy calculation that
$\omega_U(v,v)\in\mathbf{H}^2(\Omega_C^{\bullet})\cong
\mathbf{H}^2({\mathcal F}^{\bullet}_U\otimes k(x))$
is nothing but the obstruction class for the lifting of
$(E^v,\nabla^v,\{(l^v)^{(i)}_j,(V^v)^{(i)}_{j,k},(L^v)^{(i)}_{j,k},(\pi^v)^{(i)}_{j,k},(\phi^v)^{(i)}_{j,k}\})$
to a $\mathbb{C}[t]/(t^3)$-valued point of
$M^{\rm spl}_{C,D}((r^{(i)}_j),a)_{\bnu}$.
Since $M^{\rm spl}_{C,D}((r^{(i)}_j),a)_{\bnu}$ is smooth,
we have $\omega_U(v,v)=0$.
So $\omega_U$ is skew-symmetric and we can regard that
$\omega_{M^{\rm spl}_{C,D}((r^{(i)}_j),a)_{\bnu}}\in H^0(\Omega^2_{M^{\rm spl}_{C,D}((r^{(i)}_j),a)_{\bnu}})$.

\begin{theorem} \label {thm:symplectic-generic-ramified}
Take $\bnu=(\nu^{(i)}_j(w^{(i)}_j))\in N((r^{(i)}_j),a)$
and write
\[
 \nu^{(i)}_j(w^{(i)}_j)=\sum_{l=0}^{m_ir^{(i)}_j-r^{(i)}_j} c^{(i)}_{j,l}\, (w^{(i)}_j)^l \frac{dz_i}{z_i^{m_i}}
 \quad (c^{(i)}_{j,l}\in\mathbb{C}).
\]
We assume that $c^{(i)}_{j,1}\neq 0$ for any $i,j$.
Then the $2$-form $\omega_{M^{\rm spl}_{C,D}((r^{(i)}_j),a)_{\bnu}}$ 
defined in (\ref {equation: symplectic form}) is symplectic,
that is, $d\omega_{M^{\rm spl}_{C,D}((r^{(i)}_j),a)_{\bnu}}=0$ and
$\omega_{M^{\rm spl}_{C,D}((r^{(i)}_j),a)_{\bnu}}$
is non-degenerate.
\end{theorem}

\begin{proof}
We take a regular parameter $z_i$ of ${\mathcal O}_{C,t_i}$
as a function on an affine open neighborhood $U_i$ of $t_i$.
We may assume that $t_j\notin U_i$ for $j\neq i$.
We choose 
\[
 (\kappa_{i,0},\ldots,\kappa_{i,m_i-1})\in\mathbb{C}^{m_i}
\]
satisfying $\kappa_{i,0}=0$ and $\kappa_{i,q}\neq\kappa_{i,q'}$
for $q\neq q'$.
If we take a canonical parameter $h$ of $\Spec\mathbb{C}[h]$,
then
$z_i-h\kappa_{i,q}$ becomes a function on $U_i\times\Spec\mathbb{C}[h]$
and we can consider its zero scheme 
\[
 D_{i,q}\subset C\times\Spec\mathbb{C}[h].
\]
If we take some affine open neighborhood $H$ of $0$ in $\Spec\mathbb{C}[h]$,
we may assume that the fibers $\{(D_{i,q})_h|1\leq i\leq n,0\leq q\leq m_i-1\}$
are distinct $\sum_{i=1}^nm_i$ points for $h\neq 0$.
We set 
\begin{align*}
 \tilde{D}_i
 &:=
 \sum_{q=0}^{m_i-1}(D_{i,q})_H
 \\
 \tilde{D}
 &:=
 \sum_{i=1}^n\tilde{D}_i,
\end{align*}
where $(D_{i,q})_H$ means the base change of $D_{i,q}$ by $H\hookrightarrow\Spec\mathbb{C}[h]$.
If we write $a_{i,q}:=h\kappa_{i,q}$ and $z_{i,q}:=z_i-a_{i,q}$,
then $z_{i,q}$ becomes a local defining equation of $(D_{i,q})_H$.
By construction, we have $(D_{i,0})_H=t_i\times H$ for any $i$.
We put
$T:=\Spec\mathbb{C}[t]\times H$.
Set
\begin{align*}
 A^{(i)}_j
 &:=
 {\mathcal O}_{\tilde{D}_i\times T}[W^{(i)}_j] \big/
 \big( (W^{(i)}_j)^{r^{(i)}_j}-t^{r^{(i)}_j}-z_i \big)
 \\
 A^{(i)}_{j,k} &:=A^{(i)}_j\Big/\Big( z_{i,1}\cdots z_{i,m_i-1}\big(W^{(i)}_j-\zeta_{r^{(i)}_j}^k t\big) \Big),
\end{align*}
where $\zeta_{r^{(i)}_j}$ is a primitive $r^{(i)}_j$-th root of unity.
Note that
\[
 \prod_{l=0}^{r^{(i)}_j-1}(W^{(i)}_j-t\zeta_{r^{(i)}_j}^l)
 =(W^{(i)}_j)^{r^{(i)}_j}-t^{r^{(i)}_j}=z_i
\]
and
$A^{(i)}_j\otimes k(x_0)\cong\mathbb{C}[w^{(i)}_j]/((w^{(i)}_j)^{m_ir^{(i)}_j})$,
for the closed point $x_0$ of $T$ corresponding to
$h=0$ and $t=0$.
We put
\[
 \tilde{\nu}^{(i)}_j(W^{(i)}_j)
 =
 \sum_{l=0}^{m_ir^{(i)}_j-r^{(i)}_j} c^{(i)}_{j,l} \, (W^{(i)}_j)^l
 \dfrac{d z_i}{z_{i,0}z_{i,1}\cdots z_{i,m_i-1}}
 \in A^{(i)}_j \otimes\Omega^1_{C\times T/T}(\tilde{D}_T).
\]
If we take some quasi-finite dominant morphism $T'\longrightarrow T$ with $T'$ normal,
we may assume that there are $b^{(j)}_{i,0},\ldots,b^{(j)}_{i,m_i-1}\in{\mathcal O}_{T'}$
satisfying $(b^{(j)}_{i,q})^{r^{(i)}_j}=t^{r^{(i)}_j}+a_{i,q}$.
Replacing $T'$ by its shrinking, we may assume that
\begin{equation} \label {equation: genericity condition of exponent}
 \sum_{l=1}^{m_ir^{(i)}_j-r^{(i)}_j}
 c^{(i)}_{j,l} \cdot (\zeta_{r^{(i)}_j}^{lk}-\zeta_{r^{(i)}_j}^{lk'}) (b^{(j)}_{i,q}(u))^{l-1} \neq 0
\end{equation}
for any $u\in T'$ and $0\leq k<k'\leq r^{(i)}_j-1$.

We define a moduli functor
${\mathcal M}^{\rm spl}_{\tilde{\bnu},T'} \colon (\sch/T')^o\longrightarrow(\sets)$
from the category of noetherian schemes $(\sch/T')$ over $T'$ to the category of sets by
\[
 {\mathcal M}^{\rm spl}_{\tilde{\bnu},T'}(S)=
 \left\{ (E,\nabla,\{ l^{(i)}_j,V^{(i)}_{j,k},L^{(i)}_{j,k},\pi^{(i)}_{j,k},\phi^{(i)}_{j,k}\}) \right\}/\sim,
\]
for a noetherian scheme $S$ over $T'$,
where
\begin{itemize}
\item[(i)] 
$E$ is a vector bundle on $C\times S$ of rank $r$ and $\deg(E|_{C\times s})=a$
for any $s\in S$,
\item[(ii)] 
$\nabla \colon E\longrightarrow E\otimes\Omega^1_{C\times S/S}(\tilde{D}\times_T S)$ 
is a relative connection,
\item[(iii)] 
$E|_{(\tilde{D}_i)_S}=l^{(i)}_0\supset l^{(i)}_1\supset\cdots
\supset l^{(i)}_{s_i-1}\supset l^{(i)}_{s_i}=0$
is a filtration by ${\mathcal O}_{(\tilde{D}_i)_S}$-submodules
such that 
$\nabla|_{(\tilde{D}_i)_S}(l^{(i)}_j)\subset l^{(i)}_j\otimes\Omega^1_{C\times S/S}(\tilde{D}\times_TS)$
for any $i,j$
and that each $l^{(i)}_j/l^{(i)}_{j+1}$ is a locally free
${\mathcal O}_{(\tilde{D}_i)_S}$-module of rank $r^{(i)}_j$,
\item[(iv)] 
$l^{(i)}_j/l^{(i)}_{j+1}=V^{(i)}_{j,0}\supset V^{(i)}_{j,1}
\supset\cdots\supset V^{(i)}_{j,r^{(i)}_j-1}\supset z_{i,0}V^{(i)}_{j,0}$
is a filtration by ${\mathcal O}_{(\tilde{D}_i)_S}$-submodules
such that $V^{(i)}_{j,r^{i)}_j-1}/z_{i,0}V^{(i)}_{j,0}$ and
$V^{(i)}_{j,k}/V^{(i)}_{j,k+1}$ for $k=0,\ldots, r^{(i)}_j-1$
are locally free ${\mathcal O}_{(D_{i,0})_S}$-modules of rank one and that
$\nabla^{(i)}_j(V^{(i)}_{j,k})\subset V^{(i)}_{j,k}\otimes\Omega^1_{C\times S/S}(\tilde{D}\times_TS)$
for any $k$, where
$\nabla^{(i)}_j \colon l^{(i)}_j/l^{(i)}_{j+1}\longrightarrow
l^{(i)}_j/l^{(i)}_{j+1}\otimes\Omega^1_{C\times S/S}(\tilde{D}\times_TS)$
is the homomorphism induced by $\nabla|_{(\tilde{D}_i)_S}$,
\item[(v)] 
$\pi^{(i)}_{j,k} \colon V^{(i)}_{j,k}\otimes A^{(i)}_{j,k}\longrightarrow L^{(i)}_{j,k}$
is a locally free quotient $A^{(i)}_{j,k}\otimes{\mathcal O}_S$-module of rank one
such that
$p^{(i)}_{j,k}:=\pi^{(i)}_{j,k}|_{V^{(i)}_{j,k}}\colon
V^{(i)}_{j,k}\longrightarrow L^{(i)}_{j,k}$ is surjective and the diagram
\[
 \begin{CD}
  V^{(i)}_{j,k}\otimes A^{(i)}_{j,k} @>\pi^{(i)}_{j,k}>> L^{(i)}_{j,k} \\
  @V\nabla^{(i)}_j\otimes \mathrm{id} VV
  @VV\tilde{\nu}^{(i)}_j(W^{(i)}_j)+\frac {k\, dz_i} {r^{(i)}_j z_{i,0}} V \\
  V^{(i)}_{j,k}\otimes\Omega^1_{C\times S/S}(\tilde{D}\times_TS)\otimes A^{(i)}_{j,k}
  @>\pi^{(i)}_{j,k}>> L^{(i)}_{j,k}\otimes\Omega^1_{C\times S/S}(\tilde{D}\times_TS)
 \end{CD}
\]
is commutative for $0\leq k\leq r^{(i)}_j-1$,
\item[(vi)]
$\phi^{(i)}_{j,k} \colon L^{(i)}_{j,k} \longrightarrow L^{(i)}_{j,k-1}$
is an $A^{(i)}_j\otimes{\mathcal O}_S$-homomorphism whose image is
$(W^{(i)}_j-\zeta_{r^{(i)}_j}^{k-1}t)L^{(i)}_{j,k-1}$ for $1\leq k\leq r^{(i)}_j-1$
and
$\phi^{(i)}_{j,r^{(i)}_j} \colon (z_{i,0})/(z_{i,0}^{m_i+1})\otimes L^{(i)}_{j,0}
 \longrightarrow L^{(i)}_{j,r^{(i)}_j-1}$
is an $A^{(i)}_j\otimes{\mathcal O}_S$-homomorphism whose image is
$(W^{(i)}_j-\zeta_{r^{(i)}_j}^{r^{(i)}_j-1}t)L^{(i)}_{j,r^{(i)}_j-1}$
such that the diagrams
\[
 \begin{CD}
  V^{(i)}_{j,k}\otimes A^{(i)}_j @>>> V^{(i)}_{j,k-1}\otimes A^{(i)}_j  \\
  @V\pi^{(i)}_{j,k} VV  @V V  \pi^{(i)}_{j,k-1} V \\
   L^{(i)}_{j,k} @>\phi^{(i)}_{j,k}>> L^{(i)}_{j,k-1}
 \end{CD}
\]
for $k=1,\ldots,r^{(i)}_j-1$ and the diagram
\[
 \begin{CD}
  (z_{i,0})/(z_{i,0}^2z_{i,1}\cdots z_{i,m_i-1})\otimes V^{(i)}_{j,0}\otimes A^{(i)}_j
  @> >> 
  V^{(i)}_{j,r^{(i)}_j-1}\otimes A^{(i)}_j
  \\
  @V\mathrm{id}\otimes\pi^{(i)}_{j,0} VV   @VV \pi^{(i)}_{j,r^{(i)}_j-1} V \\
  \big(z_{i,0}\big)\big/\big(z_{i,0}z_{i,1}\cdots z_{i,m_i-1}(W^{(i)}_j-t)\big)\otimes L^{(i)}_{j,0} 
  @>\phi^{(i)}_{j,r^{(i)}_j}>> L^{(i)}_{j,r^{(i)}_j-1}
 \end{CD}
\]
are commutative,
\item[(vii)]
$\phi^{(i)}_{j,k}\colon L^{(i)}_{j,k} \longrightarrow
(W^{(i)}_j-\zeta_{r^{(i)}_j}^{k-1}t) L^{(i)}_{j,k-1}$
factors through an $A^{(i)}_j\otimes{\mathcal O}_S$-isomorphisms
\[
 \psi^{(i)}_{j,k}\colon
 L^{(i)}_{j,k} \stackrel{\sim}\longrightarrow
 \big(W^{(i)}_j-\zeta_{r^{(i)}_j}^{k-1}t\big)\big/
 \Big( \big( W^{(i)}_j-\zeta_{r^{(i)}_j}^{k-1}t\big) \big( W^{(i)}_j-\zeta_{r^{(i)}_j}^k t\big)
 z_{i,1}\cdots z_{i,m_i-1}\Big)
 \otimes L^{(i)}_{j,k-1}
\]
for $1\leq k\leq r^{(i)}_j-1$
such that the composition
\begin{align*}
 &
 (z_{i,0})/\big( (W^{(i)}_j-t)z_{i,0}z_{i,1}\cdots z_{i,m_i-1} \big)\otimes L^{(i)}_{j,0}
 \xrightarrow{\phi^{(i)}_{r^{(i)}_j}}
 L^{(i)}_{j,r^{(i)}_j-1}
 \\
 &\xrightarrow[\sim]{\psi^{(i)}_{j,r^{(i)}_j-1}}
 \big(W^{(i)}_j-\zeta_{r^{(i)}_j}^{r^{(i)}_j-2}t\big)\big/
 \Big( \big( W^{(i)}_j-\zeta_{r^{(i)}_j}^{r^{(i)}_j-2}t\big) 
 \big( W^{(i)}_j-\zeta_{r^{(i)}_j}^{r^{(i)}_j-1} t\big)
 z_{i_1}\cdots z_{i_{m_i-1}}\Big)
 \otimes L^{(i)}_{j,r^{(i)}_j-2}
 \\
 &\xrightarrow[\sim]{\psi^{(i)}_{r^{(i)}_j-2}}\cdots\xrightarrow[\sim]{\psi^{(i)}_{j,1}}
 \Big(\prod_{k=0}^{ r^{(i)}_j-2}(W^{(i)}_j-\zeta_{r^{(i)}_j}^k t) \Big)
 \Big/ \big( z_{i,0}z_{i,1}\cdots z_{i,m_i-1}\big)
 \otimes L^{(i)}_{j,0}
\end{align*}
coincides with the homomorphism canonically induced by
\[
 (z_{i,0})/\big( (W^{(i)}_j-t)z_{i,0}z_{i,1}\cdots z_{i,m_i-1} \big)
 \longrightarrow
 \Big(\prod_{k\neq r^{(i)}_j-1}(W^{(i)}_j-\zeta_{r^{(i)}_j}^k t) \Big)
 \Big/ \big( z_{i,0}z_{i,1}\cdots z_{i,m_i-1}\big),
\]
\item[(viii)] for any point $s\in S$,
$(E,\nabla,\{l^{(i)}_j\})\otimes k(s)$ has only
scalar endomorphisms.
\end{itemize}
Here $(E,\nabla,\{l^{(i)}_j,V^{(i)}_{j,k},L^{(i)}_{j,k},\pi^{(i)}_{j,k},\phi^{(i)}_{j,k}\})\sim
(E',\nabla',\{l'^{(i)}_j,V'^{(i)}_{j,k},L'^{(i)}_{j,k},\pi'^{(i)}_{j,k},\phi'^{(i)}_{j,k}\})$
if there are a line bundle ${\mathcal L}$ on $S$ and isomorphisms
$\theta \colon E\xrightarrow{\sim}E'\otimes {\mathcal L}$,
$\vartheta^{(i)}_{j,k} \colon L^{(i)}_{j,k}\xrightarrow{\sim}L'^{(i)}_{j,k}\otimes {\mathcal L}$
satisfying $\nabla'\circ\theta=(\theta\otimes\mathrm{id})\circ\nabla$,
$\theta|_{\tilde{D}_i}(l^{(i)}_j)\subset l'^{(i)}_j\otimes {\mathcal L}$ for any $i,j$
and for the induced isomorphism
$\theta^{(i)}_j\colon l^{(i)}_j/l^{(i)}_{j+1}\xrightarrow{\sim}
l'^{(i)}_j/l'^{(i)}_{j+1}\otimes {\mathcal L}$,
$\pi'^{(i)}_{j,k}\circ\theta^{(i)}_j|_{V^{(i)}_{j,k}}
=\vartheta^{(i)}_{j,k}\circ \pi^{(i)}_{j,k}|_{V^{(i)}_{j,k}}$,
$(\phi'^{(i)}_{j,k}\otimes\mathrm{id}_{\mathcal L})\circ\vartheta^{(i)}_{j,k}
=\vartheta^{(i)}_{j,k-1}\circ\phi^{(i)}_{j,k}$
and
$(\phi'^{(i)}_{j,r^{(i)}_j}\otimes\mathrm{id}_{\mathcal L})
\circ(\mathrm{id}\otimes\vartheta^{(i)}_{j,0})
=\vartheta^{(i)}_{j,r^{(i)}_j-1}\circ\phi^{(i)}_{j,r^{(i)}_j}$.

We can see by the similar argument to that of
Proposition \ref{prop:algebraic-space-moduli}
that the \'etale sheafification of ${\mathcal M}^{\rm spl}_{\tilde{\bnu},T'}$
can be represented by an algebraic space
$M^{\rm spl}_{\tilde{\bnu},T'}$ locally of finite type over $T'$.
If we set
\[
 \lambda_i:=\sum_{j=0}^{s_i-1}\sum_{k=0}^{r^{(i)}_j-1}
 \left(\sum_{l'=0}^{m_i-1}
 c^{(i)}_{j,\,r^{(i)}_j l'} (z_i+t^{r^{(i)}_j})^{l'}\frac{dz_i}{z_{i,0}z_{i,1}\cdots z_{i,m_i-1}}
 +k\frac{dz_{i,0}}{r^{(i)}_j z_{i,0}}\right)
\]
and $\blambda:=(\lambda_i)_{1\leq i\leq n}$, then there is 
a moduli scheme
$M(1,\blambda)$ which represents the functor
\begin{align*}
 (\sch/T') & \longrightarrow (\sets) \\
 S & \mapsto
 \left\{
 (L,\nabla) \left|
 \begin{array}{l}
  \text{$L$ is a line bundle on $C\times S$} \\
  \text{$\nabla_L\colon L\longrightarrow
  L\otimes\Omega^1_{C\times S/S}(\tilde{D}_S)$
  is a relative connection} \\
  \text{and $\nabla_L|_{(\tilde{D}_i)_S}=(\lambda_i)_S$ for any $i$}
 \end{array}
 \right.\right\}.
\end{align*}
We can see that $M(1,\blambda)$ is an affine space bundle over the relative Jacobian
of $C\times T'$ over $T'$ and it is smooth over $T'$.
We can define a morphism
$\det\colon M^{\rm spl}_{\tilde{\bnu},T'}\longrightarrow M(1,\blambda)$
by $(E,\nabla)\mapsto \det(E,\nabla)$
and we can prove by the similar argument to that of Proposition \ref{prop:det-smooth}
that this is a smooth morphism.
So
$M^{\rm spl}_{\tilde{\bnu},T'}$
is smooth over $T'$.
We take a universal family $(\tilde{E},\tilde{\nabla},\{\tilde{l}^{(i)}_j,\tilde{V}^{(i)}_{j,k},\tilde{L}^{(i)}_{j,k},
\tilde{\pi}^{(i)}_{j,k},\tilde{\phi}^{(i)}_{j,k}\})$
over some quasi-finite \'etale covering $M'_{T'}$ of $M^{\rm spl}_{\tilde{\bnu},T'}$.

First we consider the fiber $(M^{\rm spl}_{\tilde{\bnu},T'})_x$ over a point $x\in T'$ satisfying $h\neq 0$
for the corresponding point of $T=\Spec\mathbb{C}[t]\times H$.
Take
$(E,\nabla,\{l^{(i)}_j,V^{(i)}_{j,k},L^{(i)}_{j,k},\pi^{(i)}_{j,k},\phi^{(i)}_{j,k}\})
\in (M^{\rm spl}_{\tilde{\bnu},T'})_x$.
Note that we have
\[
 A^{(i)}_{j,k}\otimes k(x)\cong
 {\mathcal O}_{(D_{i,0})_x}[W^{(i)}_j]/(W^{(i)}_j-\zeta_{r^{(i)}_j}^kt)\oplus
 \bigoplus_{q=1}^{m_i-1}{\mathcal O}_{(D_{i,q})_x}[W^{(i)}_j]\big/
 \big( (W^{(i)}_j)^{r^{(i)}_j}-(t^{r^{(i)}_j}+a_{iq})_x\big).
\]
Consider the case $q\neq 0$.
If $(t^{r^{(i)}_j}+a_{i,q})(x)\neq 0$,
we have a direct sum decomposition
\[
 V^{(i)}_{j,0}\otimes {\mathcal O}_{(\tilde{D}_i)_x}/(z_{i,q}(x))\cong
 \bigoplus_{k=0}^{r^{(i)}_j-1} L^{(i)}_{j,0}\otimes
 {\mathcal O}_{(D_{i,q})_x}[W^{(i)}_j]\big/\big(W^{(i)}_j-\zeta_{r^{(i)}_j}^kb^{(j)}_{i,q}(x)\big),
\]
since $(b^{(j)}_{i,q})^{r^{(i)}_j}=(t^{r^{(i)}_j}+a_{i,q})_x$.
Each component
$L^{(i)}_{j,0}\otimes{\mathcal O}_{(D_{i,q})_x}[W^{(i)}_j]\big/
\big( W^{(i)}_j-\zeta_{r^{(i)}_j}^kb^{(j)}_{i,q}(x)\big)$
is preserved by $\tilde{\nu}^{(i)}_j(W^{(i)}_j)$ and it is an eigen-space of $\nabla^{(i)}_j$.
From the definition of $T'$ given by (\ref  {equation: genericity condition of exponent}),
we can see that
\[
 \res_{z_{i,q}(x)}\left(\tilde{\nu}^{(i)}_j(W^{(i)}_j) \pmod{W^{(i)}_j-\zeta_{r^{(i)}_j}^kb^{(j)}_{i,q}(x)}\right)\neq
 \res_{z_{i,q}(x)}\left(\tilde{\nu}^{(i)}_j(W^{(i)}_j) \pmod{W^{(i)}_j-\zeta_{r^{(i)}_j}^{k'}b^{(j)}_{i,q}(x)}\right)
\]
for $0\leq k<k'\leq r^{(i)}_j-1$.
So the eigenvalues of $\nabla^{(i)}_j$ on
$V^{(i)}_{j,0}\otimes{\mathcal O}_{(\tilde{D}_i)_x}/(z_{i,q}(x))$
are mutually distinct for $q\neq 0$ and $b^{(j)}_{i,q}(x)\neq 0$.
For $q\neq 0$ with $(t^{r^{(i)}_j}+a_{i,q})(x)=0$,
we have
$b^{(j)}_{i,q}(x)=0$ and
\[
 V^{(i)}_{j,0}\otimes{\mathcal O}_{(D_{i,q})_x}\cong
 {\mathcal O}_{(D_{i,q})_x}[W^{(i)}_j]/((W^{(i)}_j)^{r^{(i)}_j})
\]
The linear map
$\res_{(D_{i,q})_x}\left(\nabla^{(i)}_j-c^{(i)}_{j,0}
\dfrac{dz_i} {z_{i,0}\cdots z_{i,m_i-1}}\right)$
on $V^{(i)}_{j,0}\otimes{\mathcal O}_{(D_{i,q})_x}$
is just the multiplication by
$\sum_{l=1}^{m_ir^{(i)}_j-r^{(i)}_j}c^{(i)}_{j,l}
\dfrac{dz_i} {z_{i,0}\cdots z_{i,m_i-1}}(W^{(i)}_j)^l$.
So we can write
\[
 \res_{(D_{i,q})_x}\left(
 \nabla^{(i)}_j|_{(D_{i,q})_x}-\frac{1}{r^{(i)}_j}\Tr(\nabla^{(i)}_j)|_{(D_{i,q})_x} \right)
 =c_1W^{(i)}_j+c_2(W^{(i)}_j)^2+\cdots+c_{r^{(i)}_j-1}(W^{(i)}_j)^{r^{(i)}_j-1}
 \]
with $c_1\neq 0$.
Note that ${\mathcal O}_{(D_{i,q})_x}[W^{(i)}_j]/((W^{(i)}_j)^{r^{(i)}_j})$-module structure
on $V^{(i)}_{j,0}\otimes{\mathcal O}_{(D_{i,q})_x}$
is equivalent to giving
the action of $c_1W^{(i)}_j+c_2(W^{(i)}_j)^2+\cdots+c_{r^{(i)}_j-1}(W^{(i)}_j)^{r^{(i)}_j-1}$
on $V^{(i)}_{j,0}\otimes{\mathcal O}_{(D_{i,q})_x}$
which is of generic nilpotent type
and is recovered by $\nabla^{(i)}_j$.

For $q=0$, we have
$L^{(i)}_{j,k}\otimes{\mathcal O}_{(D_{i,0})_x}[W^{(i)}_j]/(W^{(i)}_j-\zeta_{r^{(i)}_j}^kt)
\cong\mathbb{C}$. 
The commutative diagram
\[
 \begin{CD}
 V^{(i)}_{j,k}\otimes {\mathcal O}_{(D_{i,0})_x}
 @> \pi^{(i)}_{j,k} >>
 L^{(i)}_{j,k}\otimes{\mathcal O}_{(D_{i,0})_x}[W^{(i)}_j]/(W^{(i)}_j-\zeta_{r^{(i)}_j}^kt)
 \cong\mathbb{C} \\
 @V\nabla^{(i)}_j|_{(D_{i,0})_x} VV    
 @V V \tilde{\nu}^{(i)}_j(W^{(i)}_j)+\frac{k\, dz_i}{r^{(i)}_jz_i} V \\
 V^{(i)}_{j,k}\otimes {\mathcal O}_{(D_{i,0})_x}\otimes\Omega^1_C(D)
 @> \pi^{(i)}_{j,k}\otimes 1 >>
 L^{(i)}_{j,k}\otimes{\mathcal O}_{(D_{i,0})_x}[W^{(i)}_j]/(W^{(i)}_j-\zeta_{r^{(i)}_j}^kt)
 \otimes\Omega^1_C(D)
 \cong \mathbb{C}\cdot\dfrac{dz_i}{z_i}
 \end{CD}
\]
means that $(E,\nabla,\{V^{(i)}_{j,k},L^{(i)}_{j,k},\pi^{(i)}_{j,k}\})\otimes{\mathcal O}_{C,t_i}$
gives a local regular singular parabolic connection at $(D_{i,0})_x$ with the eigenvalues
$\res_{z_i}\left( \Big(\tilde{\nu}^{(i)}_j(W^{(i)}_j)+\frac{k\,dz_i}{r^{(i)}_j z_{i,0}} \Big)
\otimes{\mathcal O}_{(D_{i,0})_x}[W^{(i)}_j]/(W^{(i)}_j-\zeta_{r^{(i)}_j}^kt)\right)$.
Thus we can see that the fiber
$M^{\rm spl}_{\tilde{\nu},x}$ is nothing but the moduli space of
simple regular singular parabolic connections with the assumption
that the residue of the connection at $(D_{i,q})_x$ with $b_{i,q}(x)\neq 0$ and $q\neq 0$
is semi-simple of distinct eigenvalues,
the residue of the connection at $(D_{i,q})_x$ with $b_{i,q}(x)=0$ and $q\neq 0$
has a single eigenvalue $\beta$ and its minimal polynomial is
$(T-\beta)^{r^{(i)}_j}$
and regular singular parabolic structure is given by
$\big\{ V^{(i)}_{j,k}\otimes{\mathcal O}_{(D_{i,0})_x}\big\}_{0\leq k\leq r^{(i)}_j-1}$
at $(D_{i,0})_x$ .

Next we consider the fiber
$(M^{\rm spl}_{\tilde{\nu},T'})_y$ over a point $y\in T'$ satisfying
$h=0$ and $t\neq 0$ for the corresponding point of
$T=\Spec\mathbb{C}[t]\times H$.
Note that we have
\begin{align*}
 A^{(i)}_{j,k}\otimes k(y)
 &=
 {\mathcal O}_{(\tilde{D}_i)_y}[W^{(i)}_j]\big/\big( (W^{(i)}_j)^{r^{(i)}_j}-t^{r^{(i)}_j}-z \, , \:
 z^{m_i-1}(W^{(i)}_j-\zeta_{r^{(i)}_j}^kt) \big)
 \\
 &\cong
 \mathbb{C}[W^{(i)}_j]\big/\big(((W^{(i)}_j)^{r^{(i)}_j}-t^{r^{(i)}_j})^{m_i-1}(W^{(i)}_j-\zeta_{r^{(i)}_j}^kt)\big)
 \\
 &\cong
 \mathbb{C}[W^{(i)}_j]\big/\big((W^{(i)}_j-\zeta_{r^{(i)}_j}^kt)^{m_i}\big)
 \oplus
 \bigoplus_{k'\neq k} \mathbb{C}[W^{(i)}_j]\big/\big((W^{(i)}_j-\zeta_{r^{(i)}_j}^{k'}t)^{m_i-1}\big).
\end{align*}
We can see by comparing the length that the kernel of the surjection
\[
 \pi^{(i)}_{j,k}|_{V^{(i)}_{j,k}}:V^{(i)}_{j,k}\longrightarrow L^{(i)}_{j,k}
 \cong A^{(i)}_{j,k}\otimes k(y)
\]
coincides with the image of
$\ker\big( V^{(i)}_{j,r^{(i)}_j-1}\xrightarrow{\sim}L^{(i)}_{j,r^{(i)}_j-1}
\rightarrow L^{(i)}_{j,k} \big)$
via the inclusion
$V^{(i)}_{j,r^{(i)}_j-1}\hookrightarrow V^{(i)}_{j,k}$.
So we have
\[
 \ker \big( \pi^{(i)}_{j,k}\big|_{V^{(i)}_{j,k}} \big)
 \cong
 \bigoplus_{k'=k+1}^{r^{(i)}_j-1}\mathbb{C}[W^{(i)}_j]/(W^{(i)}_j-\zeta_{r^{(i)}_j}^{k'}t).
\]
Then we can see that
$\pi^{(i)}_{j,k}|_{V^{(i)}_{j,k}}$ induces the isomorphism
\[
 V^{(i)}_{j,0}\supset V^{(i)}_{j,k}\supset
 V^{(i)}_{j,k}\otimes\mathbb{C}[W^{(i)}_{j,k}]/((W^{(i)}_{j,k}-\zeta_{r^{(i)}_j}^kt)^{m_i})
 \xrightarrow{\sim}
 L^{(i)}_{j,k}\otimes\mathbb{C}[W^{(i)}_{j,k}]/((W^{(i)}_{j,k}-\zeta_{r^{(i)}_j}^k t)^{m_i})
\]
on the $(W^{(i)}_{j,k}-\zeta_{r^{(i)}_j}^k t)$-torsion parts.
So we get a direct sum decomposition
\[
 V^{(i)}_{j,0} \cong
 \bigoplus_{k=0}^{r^{(i)}_j-1} L^{(i)}_{j,k}\otimes\mathbb{C}[W^{(i)}_j]/((W^{(i)}_j-\zeta_{r^{(i)}_j}^kt)^{m_i}).
\]
On the component
$L^{(i)}_{j,k}\otimes\mathbb{C}[W^{(i)}_j]/((W^{(i)}_j-\zeta_{r^{(i)}_j}^kt)^{m_i})$,
the operation of $W^{(i)}_j$ is the same as the expansion
\begin{equation} \label {equation: r-th root of W}
 W^{(i)}_j=\zeta_{r^{(i)}_j}^k\left( t+\frac {z_i} {r^{(i)}_jt}+\cdots \right)
\end{equation}
in $z_i$,
because of the equalities
$(W^{(i)}_j)^{r^{(i)}_j}=t^{r^{(i)}_j}+z_i$
and $\big( W^{(i)}_j-\zeta_{r^{(i)}_j}^kt \big)^{m_i}=0$.
So the operation of $\tilde{\nabla}^{(i)}_j\otimes k(y)$
on $L^{(i)}_{j,k}\otimes\mathbb{C}[W^{(i)}_j]/((W^{(i)}_j-\zeta_{r^{(i)}_j}^kt)^{m_i})$
is the substitution of (\ref  {equation: r-th root of W}) in
$\tilde{\nu}^{(i)}_j(W^{(i)}_j)+k\, dz_i/r^{(i)}_jz_i$,
whose leading coefficient
is 
\[
 \sum_{l=0}^{m_ir^{(i)}_j-r^{(i)}_j} c^{(i)}_{j,l} \zeta_{r^{(i)}_j}^{lk}t^{l}.
\]
These are mutually distinct for $k=0,\ldots,r^{(i)}_j-1$,
because of the definition of $T'$ given by  (\ref  {equation: genericity condition of exponent}).
Thus the fiber $M^{\rm spl}_{\tilde{\nu},y}$ is nothing but the moduli space of
simple unramified irregular singular parabolic connections.

We define a complex ${\mathcal F}^{\bullet}_{M'_{T'}}$ by
\begin{align*}
 &{\mathcal F}^0_{M'_{T'}} =
 \left\{
 u\in{\mathcal End}(\tilde{E})
 \left|
 \begin{array}{l}
 \text{$u|_{\tilde{D}_{M'_{T'}}}(\tilde{l}^{(i)}_j)\subset\tilde{l}^{(i)}_j$ for any $i,j$
 and for the induced homomorphism} \\
 \text{$u^{(i)}_j:\tilde{l}^{(i)}_j/\tilde{l}^{(i)}_{j+1}\longrightarrow\tilde{l}^{(i)}_j/\tilde{l}^{(i)}_{j+1}$,
 $u^{(i)}_j(\tilde{V}^{(i)}_{j,k})\subset\tilde{V}^{(i)}_{j,k}$} \\
 \text{and $\tilde{\pi}^{(i)}_{j,k}\circ (u^{(i)}_j\otimes\mathrm{id})|_{\ker\pi^{(i)}_{j,k}}=0$
 for any $i,j,k$}
 \end{array}
 \right.
 \right\} \\
 &{\mathcal F}^1_{M'_{T'}} =
 \left\{
 v\in{\mathcal End}(\tilde{E})\otimes\Omega^1_{C\times M'_{T'}/M'_{T'}}((\tilde{D})_{M'_{T'}})
 \left|
 \begin{array}{l}
 \text{$v|_{\tilde{D}_{M'_{T'}}}(\tilde{l}^{(i)}_j)\subset
 \tilde{l}^{(i)}_j\otimes\Omega^1_{C\times M'_{T'}/M'_{T'}}((\tilde{D})_{M'_{T'}})$} \\
 \text{for any $i,j$ and for the induced homomorphism} \\
 \text{$v^{(i)}_j:\tilde{l}^{(i)}_j/\tilde{l}^{(i)}_{j+1}\longrightarrow
 \tilde{l}^{(i)}_j/\tilde{l}^{(i)}_{j+1}\otimes
 \Omega^1_{C\times M'_{T'}/M'_{T'}}((\tilde{D})_{M'_{T'}})$,} \\
 \text{$v^{(i)}_j(\tilde{V}^{(i)}_{j,k})\subset
 \tilde{V}^{(i)}_{j,k}\otimes\Omega^1_{C\times M'_{T'}/M'_{T'}}((\tilde{D})_{M'_{T'}})$} \\
 \text{and $(\tilde{\pi}^{(i)}_{j,k}\otimes\mathrm{id})\circ
 v^{(i)}_j|_{\tilde{V}^{(i)}_{j,k}}=0$ for any $i,j,k$}
 \end{array}
 \right.
 \right\} \\
 &\nabla_{{\mathcal F}^{\bullet}_{M'_{T'}}} \colon {\mathcal F}^0_{M'_{T'}}\ni u
 \mapsto \tilde{\nabla}u-u\tilde{\nabla}\in {\mathcal F}^1_{M'_{T'}} .
\end{align*}
Then we can see that the relative tangent bundle
$T_{M'_{T'}/T'}$ is isomorphic to
$\mathbf{R}^1(\pi_{M'_{T'}})_*({\mathcal F}^{\bullet}_{M'_{T'}})$.
If we define a pairing
\begin{gather*}
 \omega_{M'_{T'}} \colon \mathbf{R}^1(\pi_{M'_{T'}})_*({\mathcal F}^{\bullet}_{M'_{T'}})
 \times\mathbf{R}^1(\pi_{M'_{T'}})_*({\mathcal F}^{\bullet}_{M'_{T'}})
 \longrightarrow \mathbf{R}^2(\pi_{M'_{T'}})_*(\Omega^{\bullet}_{C\times M'_{T'}/M'_{T'}})
 \cong{\mathcal O}_{M'_{T'}} 
\end{gather*}
by setting
\[
 \omega_{M'_{T'}}
 \big( [(\{u_{\alpha\beta}\},\{v_{\alpha}\})] , [(\{u'_{\alpha\beta}\},\{v'_{\alpha}\})] \big)
 =
 \big[(\{\Tr(u_{\alpha\beta}\circ u'_{\beta\gamma})\},
 \{\Tr(u_{\alpha\beta}\circ v'_{\beta})-\Tr(v_{\alpha}\circ u'_{\alpha\beta})\}\big],
\]
then $\omega_{M'_{T'}}$ descends to a relative $2$-form
$\omega_{M^{\rm spl}_{\tilde{\bnu},T'}}\in
H^0(M^{\rm spl}_{\tilde{\bnu},T'},\Omega^2_{M^{\rm spl}_{\tilde{\bnu},T'}/T'})$.

The restriction of $\omega_{M^{\rm spl}_{\tilde{\bnu},T'}}$
to the fiber over the point of $T'$
corresponding to $t=0$ and $h=0$
is nothing but the $2$-form $\omega_{M^{\rm spl}_{C,D}((r^{(i)}_j),a)_{\bnu}}$
on the moduli space of parabolic connections of generic ramified type
with the exponent $\bnu$.
The restriction of $\omega_{M^{\rm spl}_{\tilde{\bnu},T'}}$
to the fiber $(M^{\rm spl}_{\tilde{\nu},T'})_x$
over a point $x$ corresponding to $h\neq 0$
coincides with the $2$-form on the moduli space
of regular singular parabolic connections defined in \cite{Inaba-1}, section 7.
By \cite[Proposition 7.2 and Proposition 7.3]{Inaba-1},
$\omega_{M^{\rm spl}_{\tilde{\bnu},T'}}\big|_{(M^{\rm spl}_{\tilde{\bnu},T'})_x}$ 
is non-degenerate and
$d\omega_{M^{\rm spl}_{\tilde{\bnu},T'}}\big|_{(M^{\rm spl}_{\tilde{\bnu},T'})_x}=0$.

Consider the restriction of $\omega_{M^{\rm spl}_{\tilde{\bnu}}}$ to the fiber over a point $y$ of $T'$
corresponding to $t\neq 0$ and $h=0$ is the $2$-form.
Then we can see that the restriction
$\omega_{M^{\rm spl}_{\tilde{\bnu},T'}}\big|_{(M^{\rm spl}_{\tilde{\bnu},T'})_y}$
coincides with the symplectic form
on the moduli space of  unramified irregular singular connections
given in \cite{Inaba-Saito}, section 4.
By \cite[Proposition 4.1]{Inaba-Saito}, 
$\omega_{M^{\rm spl}_{\tilde{\bnu},T'}}\big|_{(M^{\rm spl}_{\tilde{\bnu},T'})_y}$ 
is non-degenerate.
Since the induced homomorphism
\[
 \det(\omega_{M^{\rm spl}_{\tilde{\bnu}}}) \colon
 \det(T_{M^{\rm spl}_{\tilde{\bnu}}/T'})\longrightarrow
 \det(\Omega^1_{M^{\rm spl}_{\tilde{\bnu}}/T'})
\]
is isomorphic in codimension one, it is isomorphic on whole $M^{\rm spl}_{\tilde{\bnu}}$,
because $M^{\rm spl}_{\tilde{\bnu},T'}$ is normal.
Thus the relative $2$-form
$\omega_{M^{\rm spl}_{\tilde{\bnu},T'}}$ is non-degenerate.
The non-degeneracy and $d$-closedness of $\omega_{M^{\rm spl}_{\tilde{\bnu},T'}}$
imply the same properties for 
$\omega_{M^{\rm spl}_{C,D}((r^{(i)}_j),a)_{\bnu}}=
\omega_{M^{\rm spl}_{\tilde{\bnu},T'}}\big|_{M^{\rm spl}_{C,D}((r^{(i)}_j),a)_{\bnu}}$.
\end{proof}

Restriction of the symplectic form $\omega_{M^{\rm spl}_{C,D}((r^{(i)}_j),a)_{\bnu}}$ to the moduli space
$M^{\balpha}_{C,D}((r^{(i)}_j),a)_{\bnu}$ gives the following corollary.

\begin{corollary}
Take an exponent
$\bnu\in N((r^{(i)}_j),a)$ such that the
$\dfrac{w^{(i)}_j dw^{(i)}_j}{(w^{(i)}_j)^{m_ir^{(i)}_j-r^{(i)}_j+1}}$-coefficient of $\nu^{(i)}_{j,k}$
does not vanish for any $i,j$.
Then there is an algebraic symplectic form $\omega_{M^{\balpha}_{C,D}((r^{(i)}_j),a)_{\bnu}}$
on the moduli space
$M^{\balpha}_{C,D}((r^{(i)}_j),a)_{\bnu}$ of $\balpha$-stable parabolic connections
of generic ramified type with the exponent $\bnu$.
\end{corollary}

\section*{Appendix}

In the earlier version of the preprint, the author missed the condition (v) of
Definition \ref {definition: ramified structure} in the formulation of
irregular singular connection of generic ramified type.
Indeed,  we cannot remove the condition because of the following example.

Consider the case $m=r=2$
and a connection
$(E,\nabla)$ with
$(\hat{E},\hat{\nabla})\cong
(\mathbb{C}[[w]],\nabla_{\nu(w)})$,
where $w=\sqrt{z}$, $\nu(w)=w\, dw/w^3$ and
\[
 \nabla_{\nu(w)} \colon
 \mathbb{C}[[w]] \ni f(w) \ \mapsto \ df(w)+\nu(w)f(w)
 \in \mathbb{C}[[w]]\otimes\frac{dz}{z^2}.
\]
Let $e_0,e_1$ be the basis of $\hat{E}$ corresponding to $1,w\in\mathbb{C}[[w]]$.
Then a true ramified structure
$(V_k,L_k,\pi_k,\phi_k)$, in the sense of Definition \ref {definition: ramified structure}, is given by
$V_k=(w^k)/(w^4)$,
$L_k=(w^k/w^{3+k})$ for $k=0,1$ and
\begin{align*}
 \pi_0 & \colon \
 V_0\otimes_{\mathbb{C}[z]/(z^2)}\mathbb{C}[w]/(w^3)
 \ \ni \ f_0(w)e_0+f_1(w)e_1 \ \mapsto \ f_0(w)+f_1(w)w \ \in \ \mathbb{C}[w]/(w^3)
 \\
 \pi_1 & \colon \
 V_1\otimes_{\mathbb{C}[z]/(z^2)}\mathbb{C}[w]/(w^3)
 \ \ni \ f_1(w)e_1+f_2(w)ze_0
 \ \mapsto \ f_1(w)w+f_2(w)z \ \in \ (w)/(w^4).
\end{align*}

However, if we lose the condition (v) of Definition 1.2,
we can construct the following unexpected example:
Again we put $V_k=(w^k)/(w^4)$ and $L_k=(w^k)/(w^{3+k})$.
We define
\begin{align*}
 \pi'_0 & \colon
 V_0\otimes_{\mathbb{C}[z]/(z^2)}\mathbb{C}[w]/(w^3)
 \longrightarrow\mathbb{C}[w]/(w^3)
 \\
 \pi'_1 & \colon
 V_1\otimes_{\mathbb{C}[z]/(z^2)}\mathbb{C}[w]/(w^3)
 \longrightarrow (w)/(w^4),
\end{align*}
by setting
\begin{align*}
 \pi'_0(f_0(w)e_0+f_1(w)e_1)
 &=
 f_0(w)+f_1(w)(w+w^2)
 \\
 \pi'_1(f_1(w)e_1+f_2(w)ze_0)
 &=
  f_1(w)(w+w^2)+f_2(w)(z+2zw).
\end{align*}
Furthermore,  we define the $\mathbb{C}[w]$-homomorphisms
\begin{align*}
 \phi'_1
 &\colon 
 (w)/(w^4)\longrightarrow \mathbb{C}[w]/(w^3)
 \\
 \phi'_2
 &\colon 
 (z)/(w^5)\longrightarrow (w)/(w^4)
\end{align*}
by setting
\begin{align*}
 \phi'_1(w)
 &=w
 \\
 \phi'_2(z)
 &=z+2wz,
\end{align*}
which would not satisfy the condition (v) of 
Definition \ref {definition: ramified structure}.
Then we can check that the two squares in the diagram
\[
 \begin{CD}
  V_0\otimes\mathbb{C}[w]/(w^3) @>\pi'_0>> \mathbb{C}[w]/(w^3) \\
  @AAA @AA\phi'_1 A \\
  V_1\otimes\mathbb{C}[w]/(w^3) @>\pi'_1>> (w)/(w^4) \\
  @AAA @AA\phi'_2 A \\
  (z)\otimes V_0\otimes\mathbb{C}[w]/(w^3) @>\mathrm{id}\otimes\pi'_0>> (z)/(w^5)
 \end{CD}
\]
are commutative.
If we put $\nu'(w):=(w+w^2)dw/w^3$ which is different from the true exponent,
the diagram
\[
 \begin{CD}
  V_0\otimes\mathbb{C}[w]/(w^3) @> \pi'_0>> \mathbb{C}[w]/(w^3) \\
  @V \nabla|_{z^2=0}\otimes\mathrm{id} VV  @V V \nu'(w) V \\
  V_0\otimes\dfrac{dz}{z^2}\otimes\mathbb{C}[w]/(w^3) @> \pi'_0>> 
  \mathbb{C}[w]/(w^3)\otimes\dfrac{dz}{z^2} 
 \end{CD}
\]
is commutative,
because
\begin{align*}
 \pi'_0(\nabla(e_0))
 &=
 \pi'_0\Big(e_1\otimes \frac{dz}{2z^2}\Big)
 =(w+w^2)\frac{dw}{w^3}=\nu'(w)\pi'_0(e_0)
 \\
 \pi'_0(\nabla(e_1))
 &=
 \pi'_0\Big( ze_0\otimes \frac{dz}{2z^2}+e_1\otimes\frac{dz}{2z}\Big)
 =w^2\frac{dw}{w^3}+w\frac{dw}{w}=\nu'(w)(w+w^2)=\nu'(w)\pi'_0(e_1). 
\end{align*}
Similarly, the diagram
\[
 \begin{CD}
  V_1\otimes\mathbb{C}[w]/(w^3) @> \pi'_1>> (w)/(w^4) \\
  @V \nabla|_{z^2=0}\otimes\mathrm{id} VV  @V V \nu'(w)+\frac{dw}{w} V \\
  V_1\otimes\dfrac{dz}{z^2}\otimes\mathbb{C}[w]/(w^3) @> \pi'_1>> 
  (w)/(w^4)\otimes\dfrac{dz}{z^2} 
 \end{CD}
\]
is commutative,
because
\begin{align*}
 \pi'_1(\nabla(e_1))
 &=
 \pi'_1\Big( ze_0\otimes \frac{dz}{2z^2}+e_1\otimes \frac{dz}{2z} \Big)
 \\
 &=(z+2zw)\frac{dw}{w^3}+(w+w^2)z\frac{dw}{w^3}
 \\
 &= (z+3zw)\frac{dw}{w^3}
 =(w+w^2)\left( (w+w^2)\frac{dw}{w^3}+\frac{dw}{w}\right)
 =\left(\nu'(w)+\frac{dw}{w}\right)\pi'_1(e_1)
 \\
 \pi'_1(\nabla(ze_0))
 &=
 \pi'_1\Big(ze_1\otimes \frac{dz}{2z^2}+ze_0\otimes \frac{dz}{z}\Big)
 \\
 &=
 z(w+w^2)\frac{dw}{w^3}+(z+2zw)2z\frac{dw}{w^3}
 \\
 &=zw\frac{dw}{w^3}
 =(z+2zw)(w+2w^2)\frac{dw}{w^3}
 =
 \Big( \nu'(w)+\frac{dw}{w}\Big)\pi'_1(ze_0).
\end{align*}
So $(V_k,L_k,\pi'_k,\phi'_k)$ satisfies all the other conditions of
Definition \ref {definition: ramified structure} except the condition (v).

\end{document}